\documentclass[10pt,oneside,leqno]{amsart}
\usepackage{amsxtra}
\usepackage{amsopn}
\usepackage{color}
\usepackage{amsmath,amsthm,amssymb}
\usepackage{amscd}
\usepackage{amsfonts}
\usepackage{latexsym}
\usepackage{verbatim}

\theoremstyle{plain}
\newtheorem{theorem}{Theorem}[section]
\newtheorem{definition}[theorem]{Definition}
\newtheorem{lemma}[theorem]{Lemma}
\newtheorem{proposition}[theorem]{Proposition}
\newtheorem{corollary}[theorem]{Corollary}
\newtheorem{remark}[theorem]{Remark}
\newtheorem*{example}{Example}
\newtheorem{question}[theorem]{QUESTION}
\newtheorem{remark-question}[section]{Remark-Question}
\newtheorem{conjecture}[section]{Conjecture}

\newcommand\R{{\mathbb R}}

\newcommand\frg{{\mathfrak g}}

\begin{document}

\title{$G_2$-structures
on Einstein solvmanifolds}
\author{Marisa Fern\'andez, Anna Fino and V\'ictor Manero}

\date{\today}

\maketitle


\begin{abstract} We study  the $G_2$  analogue of the Goldberg
conjecture on  non-compact  solvmanifolds.
In contrast to the almost-K\"ahler  case we prove   that a  7-dimensional
solvmanifold cannot admit any left-invariant
calibrated $G_2$-structure $\varphi$   such that   the induced metric $g_{\varphi}$  is Einstein, unless $g_{\varphi}$ is flat. We give an example of 7-dimensional solvmanifold admitting a left-invariant  calibrated $G_2$-structure $\varphi$ such that $g_{\varphi}$ is Ricci-soliton.  Moreover, we show that a $7$-dimensional (non-flat) Einstein solvmanifold $(S,g)$
cannot admit any left-invariant cocalibrated $G_2$-structure  $\varphi$ such that the induced metric $g_{\varphi} = g$.
\end{abstract}

\bigskip


\section{Introduction}

A $7$-dimensional smooth manifold $M^7$  is said to admit a $G_2$--structure if there is a reduction of the structure group of its frame bundle from $GL(7,\R)$  to the  exceptional Lie group $G_2$ which can actually be viewed naturally as a subgroup of $SO(7)$. Therefore a   $G_2$-structure determines a Riemannian metric and an orientation. In fact, one can prove  that the presence of a $G_2$-structure is equivalent to the existence of a certain
type of  a non-degenerate $3$-form $\varphi$  on the manifold.  By \cite{FG} a   manifold $M^7$  with a $G_2$-structure comes equipped with a  Riemannian metric $g$, a cross product $P$, a 3-form $\varphi$, and orientation, which satisfy the relation
$$
\varphi(X,Y,Z) = g ( P (X, Y), Z),
$$
for every vector field $X, Y, Z$.

This is exactly analogue  to the data of an almost Hermitian manifold, which comes with a  Riemannian metric, an almost complex structure $J$, a 2-form  $F$, and an orientation, which satisfy the relation $F (X, Y) = g (JX, Y)$.

 Whenever this 3-form $\varphi$  is covariantly constant
with respect to the Levi-Civita connection then the holonomy group is contained
in $G_2$ and the $3$-form $\varphi$ is closed and co-closed.

A $G_2$-structure is called {\it calibrated} if the $3$-form  $\varphi$ is closed and it  can be viewed as the  $G_2$ analogous of an  almost-K\"aher structure in  almost Hermitian geometry. By the results  in \cite{Bryant,CI} no compact $7$-dimensional manifold $M^7$ can support a calibrated $G_2$-structure $\varphi$ whose
underlying metric $g_{\varphi}$  is Einstein unless $g_{\varphi}$ has holomomy contained in $G_2$. This could be considered to be a $G_2$  analogue of the Goldberg
conjecture in almost-K\"ahler geometry. The result was generalized by R.L. Bryant to calibrated
$G_2$-structures with too tightly pinched Ricci tensor and by R. Cleyton and S. Ivanov  to calibrated $G_2$-structures with  divergence-free
Weyl tensor.

A non-compact complete Einstein (non-K\"ahler)
almost-K\"ahler manifold with negative scalar curvature  was constructed  in \cite{ADM} and  in \cite{HOS} it was shown that  it  is  an  almost-K\"ahler solvmanifold, that is,  a simply connected solvable Lie group $S$ endowed with a left-invariant almost-K\"ahler structure \cite{HOS}. In Section \ref{AKdim6} we show that in dimension six  this is the unique example of Einstein almost-K\"ahler (non-K\"ahler) solvmanifold and we  classify the  $6$-dimensional solvmanifolds admitting a left-invariant  (non-flat) K\"ahler-Einstein structure.

A natural problem is then to study the existence of calibrated  $G_2$-structures   inducing  Einstein metrics on non-compact homogeneous Einstein manifolds.
All the known examples of non-compact homogeneous Einstein manifolds belong
to the class of solvmanifolds, that is,  they are simply connected solvable Lie groups $S$ endowed with a left invariant metric (see  for instance the survey \cite{Lauret1}).  A left-invariant metric on a Lie group $S$  will be always
identified with the inner product  $\langle \cdot, \cdot \rangle$  determined on the Lie algebra $\mathfrak s$ of $S$. According to a long
standing conjecture attributed to D. Alekseevskii (see \cite[7.57]{Besse}), these might
exhaust the class of non-compact  homogeneous Einstein manifolds.

On the other hand,  Lauret in \cite{Lauret2} showed that  the  Einstein solvmanifolds  are \emph{standard}, i.e. satisfy
the following additional condition: if $\mathfrak s = \mathfrak a  \oplus \mathfrak n$   is the orthogonal decomposition
of the Lie algebra $\mathfrak s$  of $S$ with  $\mathfrak n = [ \mathfrak s, \mathfrak s]$,  then $\mathfrak a$ is abelian.

A  left-invariant Ricci-flat metric on a solvmanifold is necessarly flat \cite{AK}, but solvmanifolds can  admit    incomplete metrics with holonomy  contained in $G_2$  as shown in   \cite{GLPS,CF}.

In Section \ref{sectioncalibratedstruct}   by using the classification of $7$-dimensional Einstein solvmanifolds and some obstructions to the existence of calibrated $G_2$-structures, in contrast to the almost-K\"ahler  case,  we prove   that a  7-dimensional  solvmanifold cannot admit any left-invariant
calibrated $G_2$-structure $\varphi$   such that   the induced metric $g_{\varphi}$  is Einstein, unless $g_{\varphi}$ is flat.

If $\varphi$ is co-closed, then the $G_2$-structure is called {\it cocalibrated}.   In Section  \ref{sectioncocalibratedstruct}  we show that a $7$-dimensional (non-flat) Einstein solvmanifold $(S,g)$
cannot admit any left-invariant cocalibrated $G_2$-structure  $\varphi$ such that the induced metric $g_{\varphi} = g$.


\section{Preliminaries on Einstein solvmanifolds}

%

%

By  \cite{Lauret2}  all  the  Einstein  solvmanifolds  are   standard.
Standard Einstein solvmanifolds constitute a distinguished class that has been
deeply studied  by J. Heber, who has obtained  many remarkable structural and
uniqueness results, by assuming only the standard condition (see \cite{Heber}). In contrast to the compact case, a standard Einstein
metric is unique up to isometry and scaling among left-invariant  metrics  \cite[Theorem E]{Heber}.  The study of standard Einstein solvmanifolds  can be reduced  to the rank-one case, that
is,  to the ones  with $\dim {\mathfrak a}  = 1$ (see \cite[Sections 4.5,4.6]{Heber}) and everything is determined by the nilpotent Lie algebra ${\mathfrak n} = [{\mathfrak s}, {\mathfrak s}].$ Indeed, a nilpotent
Lie algebra ${\mathfrak n}$  is the nilradical of a rank-one Einstein solvmanifold if and only if $\mathfrak n$ admits a nilsoliton metric (also called a minimal metric), meaning that its Ricci
operator is a multiple of the identity modulo a derivation of ${\mathfrak n}$.

Any standard Einstein solvmanifold is isometric to a solvmanifold whose
underlying metric Lie algebra resembles an Iwasawa subalgebra of a semisimple
Lie algebra in the sense that $ad_A$  is symmetric and nonzero for any $A \in \mathfrak a$, $A \neq 0$.
Moreover, if $H$ denotes the mean curvature vector of  $S$ (i.e.,  the only element $H \in {\mathfrak a}$  such that ${\rm {tr}} (ad_A) = \langle A, H\rangle$, for every $A \in {\mathfrak a}$),
 then the
eigenvalues of  $ad_H \vert_{\mathfrak n}$ are all positive integers without a common divisor, say
 $k_1 < \ldots  < k_r$. If  $d_1, \ldots, d_r$  denote the corresponding multiplicities, then the tuple
 $$
(k; d) = (k_1 < \ldots  < k_r; d_1, \ldots, d_r)
 $$
is called the \emph{eigenvalue type} of S. It turns out that $\R H \oplus {\mathfrak n}$  is also an Einstein
solvmanifold (with inner product  the restriction of $\langle \cdot, \cdot \rangle$ on it). It is thus enough to consider
 rank-one  (i.e. $\dim \mathfrak a =1$) metric solvable Lie algebras since  every higher rank Einstein
solvmanifold will correspond to a unique  rank-one  Einstein solvmanifold and to a certain
abelian subalgebra of derivations of $\mathfrak n$  containing $ad_H$.
In every dimension, only finitely many eigenvalue types occur.

 By \cite[Lemma 11]{Nikonorov}, \cite{A2}) and  \cite[Proposition 6.12]{Heber}  it follows that  if  $(\mathfrak s, \langle \cdot,  \cdot \rangle)$  is an  Einstein (non-flat) solvable Lie algebra, such
that $\dim  \mathfrak a= m$ and  $[\mathfrak s, \mathfrak s]$  is abelian, then
the eigenvalue type  is $(1; k)$, with   $k  = \dim  [\mathfrak s, \mathfrak s] \geq m$.

In the case that $\mathfrak n$ is non abelian, it is proved in \cite{Lauret3}  that any nilpotent Lie algebra of dimension $\leq 5$  admits an Einstein solvable
extension. In  \cite{Will}   it is shown that the same is true for any of the 34 nilpotent Lie algebras
of dimension 6, obtaining then a classification of all $7$-dimensional  rank-one Einstein solvmanifolds (see Table 2). A classification of  $6$ and $7$-dimensional Einstein solvmanifolds of higher rank can be obtained by \cite{Will2}, where more in general there is a study  of Ricci solitons up to dimension $7$ on solvmanifolds.
We recall that a Riemannian manifold $(M,g)$ is called Ricci soliton if the metric $g$ is such that
$Ric(g)=\lambda g+L_Xg$
for some $\lambda \in \mathbb{R}$, and $X \in \mathfrak{X}(M)$. Ricci solitons are called expanding,
steady, or shrinking depending on whether  $\lambda  < 0,  \lambda = 0,$  or  $\lambda  > 0.$
Any nontrivial homogeneous Ricci soliton must be non-compact, expanding and non-gradient (see for instance \cite{Lauret3}).  Up to now, all known examples are isometric to a left-invariant
metric  $g$  on a simply connected Lie group $G$ such that  \begin{equation} \label{solvsolitoneq} Ric(g) =  \lambda I + D,\end{equation}  for some  $\lambda \in \R$ and  some derivation $D$ of  the  Lie  algebra $\frak g$ of $G$. Conversely,
any left-invariant metric  $g$ which satisfies \eqref{solvsolitoneq}   is automatically a Ricci soliton. If $G$ is solvable, these metrics are also
called solvsolitons.


 \section{Almost-K\"ahler structures} \label{AKdim6}

An almost Hermitian manifold  $(M, J, g)$  is called an almost-K\" ahler manifold if
the corresponding K\"ahler form  $F( \cdot  ,  \cdot  ) = g( \cdot  , J  \cdot)$  is a closed $2$-form.
 In this section we study the existence of  Einstein almost-K\"ahler  structures $(J, g, F)$  on $6$-dimensional  solvmanifolds.

 Along all this work, the coefficient appearing in the rank-one  Einstein extension of a
 Lie algebra will be denoted by $a$
while the coefficients of the extension up to dimension 6 for almost-K\" ahler, and up to dimension 7 for $G_2$ manifolds,
will be denoted by $b_i$.

 \begin{theorem} A $6$-dimensional  solvmanifold  $(S, g)$ admits a left-invariant  Einstein (non-K\"ahler) almost-K\"ahler metric if and only if  its  Lie algebra $(\frak s,  g)$ is isometric to the rank-two Einstein solvable Lie algebra \eqref{AKEinstein}   defined below.

 A $6$-dimensional solvmanifold $(S, g)$ admits a left-invariant K\"ahler-Einstein structure  if and only if
 the Lie algebra $(\frak s, g)$ is isometric  either to  the rank-one Einstein solvable Lie algebra $\frak k_4$ or to the  rank-two  Einstein solvable Lie algebra  \eqref{rank2KE}
or  to the  rank-three  Einstein solvable Lie algebra \eqref{KEabelianrank3};
 both Lie algebras \eqref{rank2KE} and \eqref{KEabelianrank3} are given below.
 \end{theorem}

 \begin{proof} A   $6$-dimensional  Einstein solvable Lie algebra  $(\frak s, g)$ is necessarily standard, so one has the orthogonal decomposition (with respect to $g$)
 $$
 \frak s = \frak n  \oplus \frak  a,
 $$
 with  $\frak n = [\frak s, \frak s]$ nilpotent and $\frak a$  abelian. We will consider separately the different cases according to the rank of $\frak s$, i.e., to the dimension of $\frak a$.

 If $ \dim \frak a = 1$  and  $\frak n$ is abelian, then we know by \cite[Proposition 6.12]{Heber} that $\mathfrak s$ has structure equations
 $$
(a e^{16}, a e^{26}, a e^{36}, a e^{46}, a e^{56}, 0),
$$
where $a$ is a non-zero real number.  For this Lie algebra we get that any closed 2-form $F$ is degenerate, i.e. satisfies $F^3=0$ and so it does not admit
symplectic forms.

 If   $ \dim \frak a = 1$ but $\frak n$ is nilpotent (non-abelian), then  $(\frak s, g)$  is isometric  to one of the solvable Lie  algebras  $\frak k_i $ $(i = 1, \ldots, 8)$ defined below in Table 1, endowed with the  inner product  $g$  such that the basis $\{e_1, \ldots, e_6\}$ is orthonormal.

For $\frak k_1$, $\frak k_j$, $5 \leq j \leq 8,$  we get  again  that any  closed $2$-form  $F$ is  degenerate.

\medskip

The Lie algebras  $\frak k_2$ and $\frak k_3$ admit symplectic forms.
However, one can check  that
any almost complex structure $J$ on $\frak k_i$ $(i=2,3)$
is such that $g(\cdot , \cdot )\not=F(\cdot ,J \cdot)$.

For $\frak k_4$  we get that a  symplectic form is
$$F=\mu_{1} e^{12}+\mu_{2} e^{16}+\mu_{3} e^{26}+\mu_{1} e^{34}+\mu_{4}e^{36}+\mu_{5} e^{46}+\mu_{1}
   e^{56},$$
where $\mu_{i}$ are real numbers satisfying   $\mu_{1} \neq 0$.
The  almost complex structures $J$ such that $g(\cdot , \cdot )=F(\cdot ,J \cdot)$ are given, with respect to the basis $\{e_1, \ldots, e_6\}$,  by
$$
J e_1=\pm e_2, \quad Je_3=\pm e_4, \quad  J e_5=\pm e_6,
$$
with $e_i$ the dual of $e^i$ via the inner product, and they are integrable. Therefore,
 $(J,g, F)$ are K\"ahler-Einsten structures on $\frak k_4$.

 In order to determine all  the $6$-dimensional  rank-two   Einstein solvable Lie algebras, we need first to find the  rank-one  Einstein solvable  extensions  $\mathfrak n_4 \oplus \R\langle e_5\rangle$ of the $4$-dimensional nilpotent Lie algebras ${\mathfrak n}_4$.

 Then we consider the  standard solvable Lie algebra $\mathfrak s_6 = {\mathfrak n}_4 \oplus  \mathfrak a $, with $\mathfrak a = \R\langle e_5, e_6\rangle $ abelian and such that the basis $\{e_1, \ldots, e_6\}$ is orthonormal.

If $\mathfrak{n}=[\mathfrak{s},\mathfrak{s}]$ is abelian and dimension of $\mathfrak{a}=2$, we have to consider the structure equations

 $$
\left \{ \begin{array}{l} d e^1 = a e^{15} + b_1 e^{16},\\[2 pt]
d e^2 = a e^{25} +  b_2 e^{26},\\[2 pt]
d e^3 = a e^{35} +  b_3 e^{36},\\[2 pt]
d e^4 = a e^{45}  + b_4 e^{46},\\[2 pt]
d e^5 = d e^6  =0
\end{array} \right .
$$
and then to impose that the  inner product  for which $\{e_1, \ldots, e_6\}$ is orthonormal  has to be Einstein and $d^2 e^j =0$, $j = 1, \ldots, 6$. Solving  these  conditions we find that the structure equations are:
$$
\left \{ \begin{array}{l}
d e^1 =ae^{15}+b_1 e^{16},\\[2 pt]
d e^2 =ae^{25}+\left(-b_1-b_3-b_4\right) e^{26},\\[2 pt]
d e^3 = ae^{35}+b_3 e^{36} ,\\[2 pt]
d e^4 = ae^{45}+b_4 e^{46},\\[2 pt]
d e^5 = d e^6 =0.
\end{array} \right.
$$
where   $a=\frac{\sqrt{2(b_1^2+b_3^2+b_4^2+b_1b_3+b_1b_4+b_3b_4)}}{2}$. This Lie algebra does not admit any symplectic form.

If $\mathfrak{n}$ is nilpotent (non-abelian) and $ \dim \mathfrak{a}=2$, two cases should be  considered  for ${\mathfrak n}$ : $(0,0,e^{12},0)$
and $(0,0, e^{12},e^{13})$.
We find that they have the following rank-one Einstein solvable extensions
$$
 \begin{array}{l}
\quad (\frac 12 a e^{15}, \frac 12 a e^{25}, \frac 14 \sqrt{22} a e^{12} + a e^{35}, \frac 34 a e^{45}, 0),\\
\end{array}
$$
if $\mathfrak n=(0,0,e^{12},0)$; and
$$
 \begin{array}{l}
\quad (\frac 14 a e^{15}, \frac 12 a e^{25}, \frac 12 \sqrt{5} a e^{12} + \frac 34 a e^{35}, \frac 12 \sqrt{5} a e^{13} + a e^{45},0),\\[3 pt]
\end{array}
$$
 if $\mathfrak n=(0,0, e^{12},e^{13})$.
Now, to compute the rank-two Einstein extension of $\mathfrak{n}=(0,0,e^{12},0)$ we should consider the Lie algebra

\begin{equation*}
 \left \{
\begin{array}{l}
d e^1 =  \frac 12 ae^{15}+b_1 e^{16}+b_2 e^{26}+b_3 e^{36}+ b_4 e^{46},\\[2 pt]
d e^2 =  \frac 12 ae^{25}+b_5 e^{16}+b_6 e^{26}+b_7 e^{36}+ b_8 e^{46},\\[2 pt]
d e^3 =  \frac 14 \sqrt{22}a e^{12}+  ae^{35}+b_9 e^{16}+b_{10} e^{26}+b_{11} e^{36}+ b_{12} e^{46},\\[2 pt]
d e^4 = \frac 34 ae^{45}+b_{13} e^{16}+b_{14} e^{26}+b_{15} e^{36}+ b_{16} e^{46},\\[2 pt]
d e^5 = d e^6 =0.
\end{array}
\right.
\end{equation*}
 Then we have to  impose  the Jacobi identity  and   that  the inner product,  such that  the  basis $\{ e_1, \ldots, e_6 \}$  is orthonormal,   has to be Einstein. We  obtain  the Einstein extension:

\begin{equation}  \label{rank2KE}
\left \{ \begin{array}{l}
d e^1 =  \frac 12 ae^{15}+b_1 e^{16}+b_2 e^{26},\\[2 pt]
d e^2 =  \frac 12 ae^{25}+b_2 e^{16}+b_{10} e^{26},\\[2 pt]
d e^3 =  \frac 14 \sqrt{22}a e^{12}+  ae^{35}+(b_1+b_{10}) e^{36},\\[2 pt]
d e^4 = \frac 34 ae^{45}-2(b_{1}+ b_{10}) e^{46},\\[2 pt]
d e^5 = d e^6 =0,
\end{array}
\right.
\end{equation}
where $a = \frac{4\sqrt{66}}{33}\sqrt{3 b_1^2+5 b_{10} b_1+b_2^2+3 b_{10}^2} $, which admits the K\"ahler-Einstein structures  given, in terms of the orthonormal basis $\{e_1, \ldots, e_6\}$, by

$$
\begin{array}{l}
F =\mu_1\big(ae^{12}+2\sqrt{\frac{2}{11}}ae^{35}+2\sqrt{\frac{2}{11}}(b_1+b_{10})e^{36}\big)+\mu_2(ae^{15}+2b_1e^{16}+2b_2e^{26})\\
\hskip 11pt + \mu_3(2b_2e^{16}+ae^{25}+2b_{10}e^{26})+\mu_4\big(3ae^{45}-8(b_1+b_{10}e^{46})\big)+\mu_5e^{56},\\[8pt]
  J e_1=e_2, \quad Je_2=- e_1, \quad Je_3=2 \sqrt{\frac{2}{11}}e_5+ \sqrt{\frac{3}{11}}e_6, \quad Je_4= \sqrt{\frac{3}{11}}e_5-2 \sqrt{\frac{2}{11}}e_6,\\[8pt]
  Je_5=-2 \sqrt{\frac{2}{11}}e_3- \sqrt{\frac{3}{11}}e_4,  \qquad Je_6=-\sqrt{\frac{3}{11}}e_3+ 2\sqrt{\frac{2}{11}}e_4,
  \end{array}
$$

where $\mu_i$ are real parameters satisfying $(b_1+b_{10})\mu_1^2\mu_4\neq 0$. The almost complex structure $J$  is indeed complex i.e., the Nijenhuis tensor  of $J$ vanishes.

 From the rank-one Einstein solvable extension of $\mathfrak n=(0,0, e^{12},e^{13})$
we  get  the  $6$-dimensional  Einstein solvable Lie algebra of rank two:
\begin{equation}\label{AKEinstein}
\left \{ \begin{array}{l}
d e^1 = \frac{a}{4} e^{15}+\frac{3}{4}ae^{16},\\[2 pt]
d e^2 = \frac{a}{2} e^{25}-a e^{26},\\[2 pt]
d e^3 = \frac{1}{2} \sqrt{5} a e^{12}+\frac{3}{4} a e^{35}-\frac{a}{4} e^{36} ,\\[2 pt]
d e^4 = \frac{1}{2} \sqrt{5} a e^{13}+a e^{45}+\frac{a}{2} e^{46},\\[2 pt]
d e^5 = d e^6 =0,
\end{array}
\right .
\end{equation}
which admit the  Einstein (non-K\" ahler) almost-K\" ahler structure given by

$$
\begin{array}{l}
F=\mu_1(-2\sqrt{5}e^{12}-3e^{35}+e^{36})+\mu_2(\sqrt{5}e^{13}+2e^{45}+e^{46})+\mu_3(e^{15}+3e^{16})\\
\hskip 11pt + \mu_4(-e^{25}+2e^{26})+\mu_5e^{56},\\[8pt]
   Je_1=e_3, \quad Je_3=- e_1, \quad Je_2=-\frac{1}{\sqrt{5}}e_5 +\frac{2}{\sqrt{5}}e_6, \quad Je_4= \frac{2}{\sqrt{5}}e_5+ \frac{1}{\sqrt{5}}e_6,\\[8 pt]
    Je_5=\frac{1}{\sqrt{5}}e_2- \frac{2}{\sqrt{5}}e_4, \qquad Je_6=-\frac{1}{\sqrt{5}}e_2- \frac{2}{\sqrt{5}}e_4,
   \end{array}
   $$

   where $\mu_2(4\mu_1^2+\mu_2\mu_4)  \neq 0$. The almost-K\"ahler structure  is not integrable since
   $$N_J(e_1,e_2)=-\sqrt{5}ae_3, \hskip 3pt N_J(e_1,e_5)=a e_1, \hskip 3pt N_J(e_1,e_6)=-2a e_1.$$

Now for the rank-three extensions we proceed as  for  the previous ones.

If dim $\mathfrak{a}=3$ and $\mathfrak{n}$ is abelian, we have the Einstein solvable Lie algebra

\begin{equation}\label{KEabelianrank3}
\left \{ \begin{array}{l}
d e^1 =a e^{14}-\frac{\sqrt{6}}{2} a e^{15}+\frac{\sqrt{2}}{2} e^{16},\\[2 pt]
d e^2 = a e^{24}+\frac{\sqrt{6}}{2}a e^{25}+\frac{\sqrt{2}}{2}ae^{26},\\[2 pt]
d e^3 = a e^{34}-\sqrt{2}a e^{36} ,\\[2 pt]
d e^4 = d e^5 = d e^6 =0,
\end{array} \right.
\end{equation}
which admits the almost-K\"ahler structure given by

$$
\begin{array}{l}
F=\mu_1(\sqrt{2}e^{14}-\sqrt{3}e^{15}+e^{16})+\mu_2(\sqrt{2}e^{24}+\sqrt{3}e^{25}+e^{26})+\mu_3(-e^{34}+\sqrt{2}e^{36})\\
\hskip 11pt + \mu_4e^{45}+\mu_5e^{46}+\mu_6e^{56},\\[8 pt]
   Je_1=\frac{1}{\sqrt{3}}e_4 -\frac{1}{\sqrt{2}}e_5+\frac{1}{\sqrt{6}}e_6, \quad Je_2=\frac{1}{\sqrt{3}} e_4 +\frac{1}{\sqrt{2}}e_5 + \frac{1}{\sqrt{6}}e_6, \quad
    Je_3=-\frac{1}{\sqrt{3}} e_4+  \sqrt{\frac{2}{3}}e_6,\\[8 pt]
     Je_4=-\frac{1}{\sqrt{3}} e_1 -\frac{1}{\sqrt{3}} e_2 +\frac{1}{\sqrt{3}}e_3, \quad Je_5= \frac{1}{\sqrt{2}}e_1  -\frac{1}{\sqrt{2}}e_2 , \quad Je_6= -\frac{1}{\sqrt{6}} e_1 -\frac{1}{\sqrt{6}} e_2 -\sqrt{\frac{2}{3}}e_3,
    \end{array}
   $$

where $\mu_1\mu_2\mu_3\neq0$ and actually the almost complex structure is complex i.e., $N_J=0$.

If dim $\mathfrak{a}=3$ and $\mathfrak{n}$ is nilpotent (non-abelian) $\mathfrak{n}$ is exactly $\mathfrak{h}_3$ (the 3-dimensional Heisenberg Lie algebra), having structure equations:

$$
 \begin{array}{l}
 \quad (0,0,e^{12}).\\[3 pt]
\end{array}
$$
We find the following rank-one Einstein solvable  extension
$$
 \begin{array}{l}
 \quad (\frac{a}{2} e^{14}, \frac{a}{2} e^{24}, a e^{12} + a e^{34}, 0),\\
\end{array}
$$

proceeding in the same way as in the previous examples we find that $\mathfrak{h}_3$ does not admit a rank-three Einstein solvable extension unless it is flat.
\end{proof}

  \centerline{
  \begin{tabular}{|c|c|}\hline
$\frak s_6$ &  $6$-dimensional Einstein solvable Lie algebras of  rank one  \\\hline
    $\frak k_1$ &{ \tiny{$ (\frac{2}{13} a e^{16}, \frac{2}{13} a e^{26}, \frac{10}{13} \sqrt{3} a e^{12} + \frac{11}{13} a e^{36}, \frac{20}{13} a e^{13} + a e^{46},  \frac{10}{13} \sqrt{3} a e^{14} + \frac{15}{13} a e^{56},0)$}}\\\hline
  $\frak k_2$ &{ \tiny{$(\frac{1}{4}a e^{16}, \frac{1}{4} a e^{26}, \frac 14 \sqrt{30} a e^{12} + \frac 34 a e^{36},  \frac 14 \sqrt{30} a e^{13} + a e^{46}, \frac 12 \sqrt{5} a e^{14}+  \frac 12 \sqrt{5} a e^{23} + \frac 54 a e^{56},0)$}}\\\hline
$\frak k_3$ &{ \tiny{$ (\frac{3}{10} a e^{16}, \frac 25 a e^{26}, \frac 35 a e^{36}, \frac 15 \sqrt{30} a e^{12} + \frac{7}{10} a e^{46}, \frac{1}{5} \sqrt{30} a e^{23}+ \frac{1}{5} \sqrt{15} a e^{14} + a e^{56},0)$}}\\\hline
$\frak k_4$ & {\tiny{ $(\frac 12 a e^{16},\frac 12 a e^{26},\frac 12 a e^{36}, \frac 12 a e^{46}, a e^{12}+ a e^{34} + a e^{56})$}}\\\hline
$\frak k_5$ &{ \tiny{$ (\frac 12 a e^{16}, \frac 12 a e^{26},  2 a e^{12} + a e^{36},  \sqrt{3} a e^{13} + \frac{3}{2} a e^{46},  \sqrt{3} a e^{23} + \frac{3}{2} a e^{56},0)$}}\\\hline
$\frak k_6$ & {\tiny{$ (\frac{1}{3} a e^{16},\frac{1}{2} a e^{26},\frac{1}{2} a e^{36},a_1 e^{12} + \frac 56 a e^{46}, a e^{13} + \frac 56 a e^{56},0)$}}\\\hline
$\frak k_7$ & {\tiny{$ (\frac 12 ae^{16},\frac 12 a e^{26}, \frac 12 \sqrt{7} a e^{12} + a e^{36},\frac 34 a e^{46}, \frac 34 a e^{56},0)$}}\\\hline
$\frak k_8$ &{ \tiny{$(\frac 14 a e^{16},\frac 12 a e^{26}, \frac 14 \sqrt{26} a e^{12} + \frac 34 a e^{36}, \frac 14 \sqrt{26} a  e^{13} + a e^{46}, \frac 34 a e^{56},0)$}}\\\hline
   \end{tabular}
  }

  \medskip

  \medskip

\centerline{{\bf Table  1.}  Rank-one Einstein $6$-dimensional  solvable Lie algebras}


 \section{Calibrated $G_2$-structures}\label{sectioncalibratedstruct}

  In this section we study the existence of calibrated $G_2$-structures $\varphi$  on $7$-dimensional  solvable Lie algebras  whose underlying Riemannian metric $g_{\varphi}$ is Einstein.   We will use the classification of the $7$-dimensional Einstein solvable Lie algebras and  the following  obstructions.

 \begin{lemma} \cite{ContiF} \label{firstobstruction} If there is a non zero vector $X$ in a  $7$-dimensional Lie algebra $\mathfrak{g}$ such that $(i_{X}\varphi)^3=0$ for all
closed $3$-form $\varphi\in Z^3(\mathfrak{g}^*)$, then $\mathfrak{g}$ does not admit any  calibrated $G_2$-structure.
  \end{lemma}

 \begin{lemma} \label{secondobstruction}

  Let $\frak g$ be a  $7$-dimensional Lie algebra and $\varphi$ a $G_2$-structure on $\frak g$. Then the bilinear form $g_{\varphi}: \frak g \times \frak g \rightarrow \R$ defined by
 $$g_{\varphi}(X,Y)vol=  \frac{1}{6}(i_X{\varphi}\wedge i_Y{\varphi}\wedge \varphi)$$  has to be a Remannian metric.

  \end{lemma}

  \begin{proof} It follows by the fact that  in general  there is a $1-1$  correspondence between $G_2$-structures on a $7$-manifold
and  3-forms  $\varphi$  for which the  7-form-valued bilinear form $B_{\varphi}$   defined by
  $$
 B_{\varphi} (X, Y) = (i_X{\varphi}\wedge i_Y{\varphi}\wedge \varphi)
  $$
   is positive
definite  (see \cite{Bryant87}, \cite{Hitchin}).
    \end{proof}

  \begin{lemma} \label{thirdpobstruction}  Let  $(\frak s, g)$ be a $7$-dimensional Enstein solvable  Lie algebra endowed with  a    $G_2$-structure $\varphi$, then,    for any $A \in \mathfrak a = [\frak s, \frak s]^{\perp}$ such that $g_{\varphi} (A, A) = 1$, the forms
$$\alpha=i_{A}\varphi, \quad \beta=\varphi-\alpha\wedge A^*,$$
define an $SU(3)$-structure on $ (\R\langle A\rangle)^{\perp}$, where by $A^* \in \frak s^*$ we denote the dual of $A$. So in particular $\alpha \wedge \beta =0$ and $\alpha^3 \neq 0$.
 \end{lemma}

\begin{proof} It follows by Proposition 4.5 in \cite{Schulte}.
\end{proof}
In contrast with the almost-K\"ahler case,  we  can prove the following theorem

\begin{theorem} A $7$-dimensional  solvmanifold  cannot admit any left-invariant calibrated  $G_2$-structure $\varphi$ such that $g_{\varphi}$ is Einstein, unless $g_{\varphi}$ is flat.

  In particular,  if   the  $7$-dimensional   Einstein (non-flat)   solvmanifold   $(S,g)$ has  rank one, then $(S, g)$ hasa  calibrated $G_2$-structure if and only if  the Lie algebra  $\frak s$  of $S$ is isometric  to  the  Einstein solvable  Lie algebras $\frak g_1$, $\frak g_4$,  $\frak g_9$,
$\frak g_{18}$, $\frak g_{28}$ in Table 2.

\end{theorem}

 \begin{proof} A  $7$-dimensional  Einstein solvable Lie algebra  $(\frak s, g)$ is necessarily standard, so one has the orthogonal decomposition (with respect to $g$)
 $$
 \frak s = \frak n  \oplus \frak  a,
 $$
 with  $\frak n = [\frak s, \frak s]$ nilpotent and $\frak a$  abelian. We will consider separately the different cases according to the rank of $\frak s$, i.e., to the dimension of $\frak a$.

 If $ \dim \frak a = 1$  and  $\frak n$ is abelian, then we know by \cite[Proposition 6.12]{Heber} that $\mathfrak s$ has structure equations
 $$
(a e^{17}, a e^{27}, a e^{37}, a e^{47}, a e^{57}, a  e^{67},0),
$$
where $a$ is a non-zero real number.
Computing  the generic closed $3$-form on $\frak s$  it is easy to check that $\frak s$ cannot admit any calibrated $G_2$-structure.

 If  $ \dim \frak a = 1$ and $\frak n$ is nilpotent (non-abelian), then  $(\frak s, g)$  is isometric  to one of the solvable Lie algebras  $\frak g_i, i = 1, \ldots, 33,$ in Table 2, endowed with the  inner product  such that the basis $\{e_1, \ldots, e_7\}$ is orthonormal.
 We may apply Lemma \ref{firstobstruction} with $X = e_6$  to  all  the  Lie   algebras $\frak g_i, i = 1, \ldots, 33$,  except  to the Lie algebras $\frak g_1$, $\frak g_4$, $\frak g_9$, $\frak g_{18}$,  $\frak g_{28}$,  showing in this way  that they  do not admit any calibrated $G_2$-structure. For the  remaining Lie algebras
 $\frak g_1, \frak g_4,  \frak g_9, \frak g_{18}, \frak g_{28}$
  we   first determine the generic closed $3$-form  $\varphi$ and then, by applying Lemma \ref{thirdpobstruction}, we  impose, that  $\alpha \wedge \alpha \wedge \alpha \neq 0$ and $\alpha \wedge \beta =0$, where
\begin{equation} \label{alphabetaexpr}
 \alpha = i_{e_7} \varphi, \quad \beta = \varphi - e^7 \wedge \beta.
\end{equation}
 Moreover, we have that the  closed $3$-form $\varphi$ defines a $G_2$-structure if and only the matrix associated to the symmetric  bilinear form $g_{\varphi}$, with respect to the orthonormal basis $\{e_1, \ldots, e_7\}$,  is positive definite.
 Since the Einstein metric is unique up to scaling,  a calibrated $G_2$-structure  induces an Einstein metric if and only if the  matrix associated to  the  symmetric  bilinear form $g_{\varphi}$,  with respect to the basis $\{e_1,\dots,e_7\}$,  is a multiple of the identity matrix.
  By a direct computation we have  that  then the  Lie algebras $\frak g_1, \frak g_4,  \frak g_9, \frak g_{18}, \frak g_{28}$ admit a calibrated $G_2$-structure (see Table 3)  but they do not admit  any calibrated $G_2$-structure inducing a Einstein (non-flat)  metric.

Next, we show that result for the Lie algebra $\frak g_{28}$.
To this end, we see that any closed $3$-form $\varphi$ on $\frak g_{28}$ has the following expression:

\begin{align*}
\varphi&=\rho_{1,2,7} e^{127}-\frac{1}{2} \rho_{5,6,7} e^{136}+\rho_{2,4,7} e^{137}+\frac{1}{2} \rho_{5,6,7} e^{145}-\rho_{2,3,7} e^{147}-\rho_{2,6,7} e^{157}\\
&+\rho_{2,5,7} e^{167}+\frac{1}{2} \rho_{5,6,7}e^{235}+\rho_{2,3,7} e^{237}+\frac{1}{2} \rho_{5,6,7} e^{246}+\rho_{2,4,7} e^{247}+\rho_{2,5,7} e^{257}\\
&+\rho_{2,6,7} e^{267}+\rho_{3,4,7} e^{347}+\rho_{3,5,7} e^{357}+\rho_{3,6,7} e^{367}+\rho_{3,6,7}e^{457}-\rho_{3,5,7} e^{467}\\[2pt]
&+\rho_{5,6,7} e^{567},
\end{align*}

 where  $\rho_{i,j,k}$ are arbitrary constants denoting the coefficients of $e^{ijk}$.
\medskip

In this case, one can check that the induced metric is given by the matrix $G$ with elements
$$
\begin{array}{l}
g (e_1, e_1) =
 -\frac{1}{4} \rho_{1,2,7} \rho_{5,6,7}^2,  \quad g (e_1, e_2) = 0, \quad  g (e_1, e_3)  =  \frac{1}{4} \rho_{2,3,7} \rho_{5,6,7}^2,\\[4 pt]
  g (e_1, e_4) = \frac{1}{4} \rho_{2,4,7} \rho_{5,6,7}^2,  \quad g(e_1, e_5) = \frac{1}{4} \rho_{2,5,6} \rho_{5,6,7}^2, \quad  g(e_1, e_6) =\frac{1}{4} \rho_{2,6,7} \rho_{5,6,7}^2, \\[4 pt]
g(e_1, e_7) =0, \quad  g( e_2, e_2) = -\frac{1}{4} \rho_{1,2,7} \rho_{5,6,7}^2, \quad  g(e_2, e_3)  =  -\frac{1}{4} \rho_{2,4,7} \rho_{5,6,7}^2,\\[4 pt]
 g(e_2, e_4) = \frac{1}{4} \rho_{2,3,7} \rho_{5,6,7}^2,  \quad g(e_2, e_5) = \frac{1}{4} \rho_{2,6,7} \rho_{5,6,7}^2,  \quad g (e_2, e_6) = -\frac{1}{4} \rho_{2,5,7}
   \rho_{5,6,7}^2,
   \end{array}
   $$
   $$
   \begin{array}{l}
 g(e_2, e_7) =0, \quad g(e_3, e_3) =  -\frac{1}{4} \rho_{3,4,7} \rho_{5,6,7}^2, \quad  g(e_3, e_4) = 0,  \quad g(e_3, e_5) = \frac{1}{4} \rho_{3,6,7} \rho_{5,6,7}^2,\\[4 pt]
  g(e_3, e_6) = -\frac{1}{4} \rho_{3,5,7}
   \rho_{5,6,7}^2,  \quad g(e_3, e_7) = 0,  \quad
    g(e_4, e_4) =  -\frac{1}{4} \rho_{3,4,7} \rho_{5,6,7}^2,\\[4 pt]
    g(e_4, e_5) =  -\frac{1}{4} \rho_{3,5,7} \rho_{5,6,7}^2,  \quad g(e_4, e_6) =  -\frac{1}{4} \rho_{3,6,7}
   \rho_{5,6,7}^2,  \quad g(e_4, e_7) = 0,\\
 g(e_5, e_5) = \frac{\rho_{5,6,7}^3}{4},  \quad g(e_5, e_6) =0,  \quad g(e_5, e_7) =0, \quad g(e_6, e_6) =  \frac{\rho_{5,6,7}^3}{4},
 \quad g(e_6, e_7) = 0, \\[4 pt]
 g (e_7, e_7) =-\rho_{5,6,7} \rho_{2,3,7}^2+\rho_{1,2,7} \rho_{3,5,7}^2+\rho_{1,2,7} \rho_{3,6,7}^2+
 \rho_{2,5,7}^2 \rho_{3,4,7}+\rho_{2,6,7}^2 \rho_{3,4,7}\\[4 pt]
 \hskip 47 pt +\rho_{2,5,7} \left(2 \rho_{2,3,7} \rho_{3,6,7}-2 \rho_{2,4,7}
   \rho_{3,5,7}\right)
-2 \rho_{2,6,7} \left(\rho_{2,3,7} \rho_{3,5,7}+\rho_{2,4,7} \rho_{3,6,7}\right)\\[4 pt]
 \hskip 47 pt -\rho_{2,4,7}^2 \rho_{5,6,7}+\rho_{1,2,7} \rho_{3,4,7} \rho_{5,6,7}.
\end{array}
 $$
 \smallskip

Now we have that
the system
$G=k\cdot I_7$ does not have solution, for any real number $k$, where $I_7$
is the identity matrix.
This means that the Lie algebra $\frak g_{28}$ does not admit any calibrated $G_2$-structure
defining an Einstein metric.
However, we can solve $48$ from the $49$ equations of
the system $G=k\cdot I_7$, and we obtain the metric defined by the matrix
$$G=2\left(
\begin{array}{ccccccc}
 1 & 0 & 0 & 0 & 0 & 0 & 0 \\
 0 & 1 & 0 & 0 & 0 & 0 & 0 \\
 0 & 0 & 1 & 0 & 0 & 0 & 0 \\
 0 & 0 & 0 & 1 & 0 & 0 & 0 \\
 0 & 0 & 0 & 0 & 1 & 0 & 0 \\
 0 & 0 & 0 & 0 & 0 & 1 & 0 \\
 0 & 0 & 0 & 0 & 0 & 0 & 4
\end{array}
\right).
$$

Since this matrix is positive definite,
the Lie algebra $\frak g_{28}$ has a calibrated $G_2$ form
$$\varphi=-2 e^{127}-2 e^{347}-e^{136}+e^{145}+e^{235}+e^{246}+2 e^{567},$$
which induces the metric defined by $G$.

 In order to determine  all the $7$-dimensional  rank-two  Einstein solvable Lie algebras, we need first to find the rank-one Einstein solvable  extensions  $\frak s_6 =  \mathfrak n_5 \oplus \R\langle e_6\rangle$ of any of the eight  $5$-dimensional nilpotent Lie algebras ${\mathfrak n}_5$ (see Table 1)
 and then consider the  standard solvable Lie algebra $\mathfrak s_7 = {\mathfrak n}_5 \oplus  \mathfrak a $, with $\mathfrak a = \R\langle e_6, e_7 \rangle $ abelian and such that the basis $\{e_1, \ldots, e_7\}$ is orthonormal.
From $\frak k_1$ we get the $7$-dimensional  Einstein Lie algebra of rank two with structure equations
$$
\left \{ \begin{array}{l}
d e^1 = \frac 13 \sqrt{6}a e^{16} + 4a e^{17},\\[2 pt]
d e^2 =  \frac 32 \sqrt{6} a  e^{26} - 7 a e^{27},\\[2 pt]
d e^3 = \frac 53 \sqrt{18}  a  e^{12} + \frac{11}{6}   \sqrt{6} a e^{36} - 3 a e^{37},\\[2 pt]
d e^4 = \frac{10}{3} \sqrt{6}a  e^{13} + \frac{13}{6} \sqrt{6} a  e^{ 46} + a e^{47},\\[2 pt]
d e^5 = \frac 53 \sqrt{18} a e^{14} + \frac 52 \sqrt{6} a  e^{56} + 5 a  e^{57},\\[2 pt]
d e^6 = d e^7 =0.
\end{array} \right.
$$
By computing the generic closed $3$-form $\varphi$ and by using Lemma \ref{secondobstruction} and Lemma \ref{thirdpobstruction}, we get that the matrix associated to $g_{\varphi}$, with respect to the  basis $\{e_1,\ldots, e_7\}$,  cannot be a multiple of the identity matrix.

From $\frak k_2$ we do not  get any  $7$-dimensional  Einstein Lie algebra of rank two. From $\frak k_3$  we get the $7$-dimensional  Einstein Lie algebra of rank two  with structure equations
$$
\left \{ \begin{array}{l}
d e^1 = \frac 17 \sqrt{21} a e^{16}  - a e^{17},\\[2 pt]
d e^2 = \frac{4}{21} \sqrt{21} a e^{26} + 2a e^{27},\\[2 pt]
d e^3 = \frac 27 \sqrt{21} a e^{36}  - 2a e^{37},\\[2 pt]
d e^4 = \frac{2}{21} \sqrt{30} \sqrt{21} a e^{12} + \frac 13 \sqrt{21}a e^{46}+a e^{47},\\[2 pt]
d e^5 = \frac{2}{21} \sqrt{30} \sqrt{21} a e^{14} + \frac{2}{21} \sqrt{15} \sqrt{21}a e^{23} + \frac{10}{21} \sqrt{21} a e^{56},\\[2 pt]
d e^6 = d e^7 =0.
\end{array} \right.
$$
By computing the generic closed $3$-form $\varphi$ and by using Lemma \ref{secondobstruction} and Lemma \ref{thirdpobstruction}, we get that the matrix associated to $g_{\varphi}$, with respect to the  basis $\{e_1,\ldots, e_7\}$,  cannot be a multiple of the identity matrix.
From $\frak k_4$  we get the $7$-dimensional  Einstein Lie algebras of rank two  with structure equations
$$
\left \{ \begin{array}{l}
d e^1 =a e^{16} - b_7 e^{17} + b_2 e^{27} + b_3 e^{37} + b_4 e^{47},\\[2 pt]
d e^2 = a e^{26}  + b_2 e^{17} + b_7 e^{27} + b_4 e^{37} - b_3 e^{47} ,\\[2 pt]
d e^3 =ae^{36} + b_3 e^{17} + b_4 e^{27} - b_{19} e^{37} + b_{14} e^{47} ,\\[2 pt]
d e^4 =a e^{46} + b_4 e^{17} - b_3 e^{27}+ b_{14} e^{37} + b_{19} e^{47},\\[2 pt]
d e^5 =2a (e^{12} + e^{34} + e^{56}) ,\\[2 pt]
d e^6 = d e^7 =0
\end{array}  \right.
$$
where $a=\frac{1}{2}\sqrt{b_7^2 + b_2^2 + 2 b_3^2 + 2 b_4^2 + b_{14}^2 + b_{19}^2}$.
We may  then apply Lemma \ref{firstobstruction} with $X = e_5$.

From $\frak k_5$  we get the $7$-dimensional  Einstein Lie algebras of rank two  with structure equations
$$
\left \{ \begin{array}{l}
d e^1 = ae^{16} +b_{19} e^{17} +b_{20} e^{27},\\[2 pt]
d e^2 = a e^{26} + b_{20} e^{17}  - b_{19} e^{27},\\[2 pt]
d e^3 =4a e^{12} +2a e^{36},\\[2 pt]
d e^4 = 2\sqrt{3} a e^{13} +  3a e^{46} +b_{19} e^{47} +b_{20} e^{57},\\[2 pt]
d e^5 = 2\sqrt{3} a e^{23} +   3a e^{56} +b_{20} e^{47} - b_{19} e^{57} ,\\[2 pt]
d e^6 = d e^7 =0.
\end{array} \right.
$$

For these Lie algebras in order to study completely all possibilities we  will study separately   the two cases
$b_{20} \neq 0$ and $b_{20} =0$.
By  computing the generic closed $3$-form $\varphi$ and using  Lemma \ref{thirdpobstruction}  we have that the system   $ g_{\varphi} (e_i, e_j) - k \delta_i^j =0$ (where $k$  is a non zero positive real number) with variables the  coefficients $c_{ijk}$ of $e^{ijk}$ in $\varphi$ has no solutions.

From $\frak k_6$  we get the two families of  $7$-dimensional  Einstein Lie algebras of rank two, namely $\frak k_{6,1}$ and $\frak k_{6,2}$  with respectively  structure equations
$$
1) \left \{ \begin{array}{l}
d e^1 = 2 ae^{16} + 2(b_{19}+b_{25})e^{17},\\[2 pt]
d e^2 = 3 ae^{26}  - (b_{19} + 2 b_{25}) e^{27} + b_{12} e^{37},\\[2 pt]
d e^3 = 3 a e^{36} + b_{12} e^{27} -  (b_{19} + 2 b_{25}) e^{37},\\[2 pt]
d e^4 = 6a e^{12}+5a e^{46} + b_{19} e^{47} + b_{12} e^{57},\\[2 pt]
d e^5 =  6a  e^{13}  + 5a e^{56} + b_{12} e^{47} + b_{25} e^{57},\\[2 pt]
d e^6 = d e^7 =0
\end{array} \right.
$$
 and
 $$
2) \left \{  \begin{array}{l}
 d e^1 = \sqrt{2} b_{25} e^{16} + 4 b_{25} e^{17},\\[2pt]
 d e^2 =  	\frac{3}{2}  \sqrt{2} b_{25} e^{26} - 3 b_{25} e^{27} - b_{12} e^{37},\\[2 pt]
 d e^3 =  	\frac{3}{2}  \sqrt{2} b_{25} e^{36} + b_{12} e^{27} -3 b_{25} e^{37},\\[2 pt]
 d e^4 =  3 \sqrt{2}  b_{25} e^{12} + \frac 52 \sqrt{2} b_{25} e^{46} + b_{25} e^{47}- b_{12} e^{57},\\[2 pt]
 d e^5 =   3 \sqrt{2}  b_{25} e^{13} + \frac 52 \sqrt{2} b_{25} e^{56}  + b_{12} e^{47} + b_{25} e^{57},\\[2 pt]
 d e^6 = d e^7 =0.
\end{array} \right.
 $$
For $1)$ we  compute first  the generic closed $3$-forms $\varphi$ and then, using  Lemma \ref{thirdpobstruction} for $A = e_7$,  we impose the condition $\alpha \wedge \beta =0$. By this condition we get in particular  that
$$
\rho_{1,2,3} \rho_{1,3,4} (b_{19} + b_{25}) =0,
$$
where by $\rho_{i,j,k}$ we denote the coefficient of $e^{ijk}$ in $\varphi$.
One can immediately exclude the case  $\rho_{1,3,4} =0$, since otherwise the element of the matrix associated to the metric $g_{\varphi}$  has to be  zero. Then we  study  separately the cases  $\rho_{1,2,3} =0$ and $b_{19} =  - b_{25}$. In both cases  we  do not find any solution for the system $ g_{\varphi} (e_i, e_j) - k \delta_i^j =0$.
 For $2)$ we study separately the cases $b_{12}b_{25}\neq0$, $b_{12} =0$ and $b_{25} =0$.

In the case $b_{12}b_{25}\neq0$ we compute first the generic closed $3$-forms $\varphi$ and then, using  Lemma \ref{thirdpobstruction} for $A = e_7$,  we impose the condition $\alpha \wedge \alpha \wedge \alpha \neq 0$, getting the
condition  $\rho_{1,2,5}\neq0$. Thus, we take the system   $ S_{ij} =g_{\varphi} (e_i, e_j) - k \delta_i^j =0$ and get the values of $\rho_{2,5,6}, \rho_{3,4,6}, \rho_{3,5,6}$ and $\rho_{2,3,6}$ from $S_{5,5}, S_{3,5}, S_{4,4}$ and $S_{3,4}$. Now $S_{3,3}=-k$, and the system does not admit any solution.

In the case $b_{12}=0$ we first  compute  the generic closed $3$-forms $\varphi$ and then we use  Lemma \ref{thirdpobstruction} for $A = e_7$,
obtaining  that $\rho_{2,3,6}, \rho_{2,4,6}, \rho_{3,5,6}$ and $\rho_{4,5,6}$ are all different from zero.

In the case $b_{25}=0$  we first  compute  the generic closed $3$-forms $\varphi$ and then we  may apply Lemma \ref{firstobstruction}
with $X=e_1,\dots,e_5$.

 From $\frak k_7$  we get the four   families of    $7$-dimensional  Einstein Lie algebras of rank two, namely $\frak k_{7,1}, \frak k_{7,2}, \frak k_{7,3}$ and $\frak k_{7,4}$  with respectively  structure equations $$
i)  \left \{ \begin{array}{l}
 d e^1 = a e^{16} + (- b_7 + b_{13})  e^{17} + b_6 e^{27},\\[2 pt]
 d e^2 = a e^{26} + b_6 e^{17} + b_7 e^{27},\\[2 pt]
 d e^3 = a   (\sqrt{7} e^{12} + 2 e^{36}) + b_{13} e^{37},\\[2 pt]
 d e^4 =  \frac{3}{2} a  e^{46} - (2 b_{13} + b_{25}) e^{47} + b_{24} e^{57},\\[2 pt]
 d e^5 =  \frac{3}{2} a e^{56} + b_{24} e^{47} + b_{25} e^{57},\\[2 pt]
 d e^6 = d e^7 =0, \quad \quad
\end{array} \right.
 $$
 where $a = \frac{2}{21}\sqrt{21 b_7^2 - 21 b_7 f_{13} + 63 b_{13}^2 + 21 b_6^2 + 42 b_{25} b_{13} + 21 b_{25}^2 + 21 b_{24}^2}$,
 $$
ii)  \left \{  \begin{array}{l}
  d e^1 =a e^{16} + (- b_7 + b_{13})  e^{17} + b_6 e^{27},\\[2 pt]
  d e^2 =a e^{26} + b_6 e^{17} + b_7 e^{27},\\[2 pt]
  d e^3 =a (\sqrt{7} e^{12} + 2 e^{36}) + b_{13} e^{37},\\[2 pt]
  d e^4  =  \frac{3}{2} a e^{46}  - b_{13} e^{47} - b_{24} e^{57},\\[2pt]
  d e^5 =    \frac{3}{2} a e^{56} + b_{24} e^{47} - b_{13} e^{57},\\[2 pt]
 d e^6 = d e^7 =0.
\end{array} \right.
 $$

where $a=\frac{2}{21}  \sqrt{21 b_7^2 - 21 b_7 b_{13} + 42 b_{13}^2 + 21 b_{6}^2} $

 $$
iii) \left \{  \begin{array}{l}
d e^1 = a e^{16} +  \frac 12 b_{13}   e^{17} - b_6 e^{27},\\[2 pt]
d e^2 = a e^{26} + b_6 e^{17} + \frac 12 b_{13} e^{27},\\[2 pt]
d e^3 = a  (\sqrt{7} e^{12} + 2 e^{36}) + b_{13} e^{37},\\[2 pt]
d e^4 = \frac{3}{2} a e^{46} - (2 b_{13} + b_{25}) e^{47} + b_{24} e^{57},\\[2 pt]
d e^5 =  \frac{3}{2}a e^{56} + b_{24} e^{47} + b_{25} e^{57},\\[2 pt]
 d e^6 = d e^7 =0.
\end{array} \right.
 $$

 where $a=\frac{1}{21} \sqrt{231  b_{13}^2 + 168 b_{13} b_{25}+ 84 b_{25}^2 + 84 b_{24}^2}$

 $$
iv)  \left \{ \begin{array}{l}
 d e^1 =  \frac 13 \sqrt{3} b_{25} e^{16} - \frac 12 b_{25} e^{17} - b_6 e^{27},\\[2 pt]
 d e^2 =  \frac 13 \sqrt{3} b_{25} e^{26} + b_6 e^{17} - \frac{1}{2} b_{25} e^{27},\\[2 pt]
 d e^3 = \frac 13 \sqrt{3} b_{25} (\sqrt{7} e^{12} + 2 e^{36}) - b_{25} e^{37},\\[2 pt]
 d e^4 = \frac 12 \sqrt{3} b_{25} e^{46} + b_{25} e^{47} - b_{24} e^{57},\\[2 pt]
 d e^5 =  \frac 12 \sqrt{3} b_{25} e^{46} + b_{24} e^{47}+ b_{25} e^{57},\\[2 pt]
   d e^6 = d e^7 =0.
\end{array} \right.
 $$
For all of them after computing the generic closed $3$-forms we may apply  Lemma \ref{firstobstruction}   with $X=e_3$.

 From $\frak k_8$  we get the   $7$-dimensional  Einstein Lie algebras of rank two  with   structure equations
$$
\left \{ \begin{array}{l}
d e^1 = a e^{16}+ (- b_{13}  + b_{19}) e^{17},\\[2 pt]
d e^2 = a e^{26} + (2 b_{13} - b_{19}) e^{27} ,\\[2 pt]
d e^3 =  a (\sqrt{26} e^{12} + 3 e^{36}) + b_{13} e^{37},\\[2 pt]
d e^4 =  a (\sqrt{26} e^{13} + 4 e^{46}) + b_{19} e^{47},\\[2 pt]
d e^5 =  3a e^{56} -(2 b_{13} +b_{19}) e^{57},\\[2 pt]
d e^6 = d e^7 =0.
\end{array} \right.
 $$

where $a=  \frac{1}{39} \sqrt{390  b_{13}^2 - 78 b_{13} b_{19} + 156 b_{19}^2}$.

We study separately the cases  $b_{13}b_{19}\neq0$, $b_{13} =0$ and $b_{19}=0$.

 In the case   $b_{13}b_{19}\neq0$  using  Lemma \ref{thirdpobstruction} (i.e. $\alpha \wedge \alpha \wedge \alpha \neq 0$, with  $A = e_ 7$) we may  suppose $\rho_{1,2,4}\rho_{1,3,5}\neq0$ for the generic closed $3$-form.  Now, we consider  the  system $S_{ij} = g_{\varphi} (e_i,e_j) - k \delta_i^j =0$.   We  take  $\rho_{1,2,5}, \rho_{1,2,3}$ and $a$ from $S_{4,5} =0= S_{3,4}$ and $S_{2,2}=0$. In the new system  we can conclude that  from equations $S_{2,4}=0$ and $S_{3,5}=0$ that there is no solution. Indeed from the two equations it follows that $b_{13}=-\frac{19}{5}b_{19}$ and $b_{13}=\frac{7}{31}b_{19}$,  which  is a contradiction since $b_{13}b_{19}\neq0$.

 In the case $b_{13}=0$ using  Lemma \ref{thirdpobstruction} we may  suppose
 $\rho_{1,3,5}\rho_{3,4,7}\neq0$ for the generic closed $3$-form. Then we get $\rho_{2,5,7}, k$ and $\rho_{2,3,7}$ from $S_{5,5} =0= S_{5,3} =S_{2,3}$. The new system satisfies:
$$S_{4,4}=\frac{7 \rho_{1,3,5} \left(49 \rho_{1,2,5}^2+152 \rho_{3,4,7}^2\right)}{76 \sqrt{78}}.$$

 In the case $f_{19}=0$ using  Lemma \ref{thirdpobstruction} we may  suppose
 $c_{1,3,5}c_{3,4,7}\neq0$ for the generic closed $3$-form. Then we get $c_{1,2,3}, c_{2,5,6}, c_{2,3,7}$ and $c_{1,3,5}$ from $S_{3,4}=0= S_{2,3}= S_{2,4}$ and $S_{3,3}=0$. The new system satisfies:
$$S_{4,4}=-\frac{959322 c_{1,2,5}^2 c_{3,4,7}^4+59711 k^2}{59711 k}.$$

what implies again $S_{4,4}\neq0$.

If $[\frak s, \frak s] = \frak n$ is abelian and $\dim \mathfrak n = 5$, we have to consider the structure equations
$$
\left \{ \begin{array}{l}
d e^1 = a e^{16} + b_1 e^{17},\\[2 pt]
d e^2 = a e^{26} + b_2 e^{27},\\[2 pt]
d e^3 = a e^{36} + b_{3} e^{37},\\[2 pt]
d e^4 = a e^{46} + b_{4} e^{47}\\[2 pt]
d e^5 = a e^{56} + b_{5} e^{57},\\[2 pt]
d e^6 = d e^7 =0.
\end{array} \right.
$$
By imposing that $\frak s$ is a Einstein  Lie algebra (the inner product  is the one for which $\{e_1, \ldots, e_7\}$  is orthonormal), we get the Lie algebras with structure equations
$$
\left \{ \begin{array}{l}
d e^1 =a e^{16} + (- b_2 - b_{3} - b_{4}) e^{17},\\[2 pt]
d e^2 =  a e^{26} + b_2 e^{27},\\[2 pt]
d e^3 =  a e^{36} + b_{3} e^{37},\\[2 pt]
d e^4 =  a e^{46} + b_{4} e^{47}\\[2 pt]
d e^5 =  a e^{56},\\[2 pt]
d e^6 = d e^7 =0,
\end{array} \right.
$$
where $a = \sqrt{10 b_7^2 + 10 b_7 b_{13} + 10 b_7 b_{19} + 10 b_{13}^2 + 10 b_{13} b_{19} + 10 b_{19}^2}.$
For these Lie algebras we first  compute  the generic closed $3$-forms $\varphi$ and then we  may apply Lemma \ref{firstobstruction}
with  $X=e_1,\dots , e_5$.

 In order to determine all  the $7$-dimensional  rank-three  Einstein solvable Lie algebras, we need first to find the rank-one Einstein solvable  extensions  $\mathfrak n_4 \oplus \R\langle e_5\rangle$ of the  two $4$-dimensional nilpotent Lie algebras ${\mathfrak n}_4$
and then consider the  standard solvable Lie algebra $\mathfrak s_7 = {\mathfrak n}_4 \oplus  \mathfrak a $, with $\mathfrak a = \R\langle e_5, e_6, e_7\rangle $ abelian and such that the basis $\{e_1, \ldots, e_7\}$ is orthonormal.
For any of the nilpotent Lie algebras ${\mathfrak n}_4$ we find the following rank-one Einstein solvable  extensions
$$
 \begin{array}{l}
1)  \quad (\frac 14 a e^{15}, \frac 12 a e^{25}, \frac 12 \sqrt{5} a e^{12} + \frac 34 a e^{35}, \frac 12 \sqrt{5}a e^{13} + a e^{45},0)\\[3 pt]
2)  \quad (\frac 12 a e^{15}, \frac 12 a e^{25}, \frac 14 \sqrt{22} a e^{12} + a e^{35}, \frac 34 a e^{45}, 0)
\end{array}
$$
From $1)$ we do not get  any  $7$-dimensional  Einstein Lie algebra of rank three. From $2)$ we get the $7$-dimensional  Einstein Lie algebra of rank three with structure equations
\begin{equation} \label{rank3nonabel}
\left \{ \begin{array}{l}
d e^1 =  \frac 12 a e^{15} - (b_{10} + \frac{1}{2} b_{28}) e^{16} + b_2 e^{26} +  (- b_{14} + b_{23}) e^{17} + b_6 e^{27},\\[2 pt]
d e^2 = \frac 12 a e^{25} + b_9 e^{16} + b_{10} e^{26} + b_{13} e^{17} + b_{14} e^{27} ,\\[2 pt]
d e^3 = \frac 14 \sqrt{22} a e^{12} + a e^{35} - \frac 12 b_{28} e^{36} + b_{23} e^{37} ,\\[2 pt]
d e^4 = \frac 34 a e^{45} + b_{28} e^{46} - 2 b_{23} e^{47},\\[2 pt]
d e^5 = d e^6 = d e^7 =0.
\end{array} \right.
\end{equation}

satisfying  the conditions $d^2e^i=0, i=1, \dots, 4$, and one of  the following:
\begin{enumerate}
\item[(i)] $a = \sqrt{\frac{32 b_{14}^2 - 32 b_{14} b_{23} + 96 b_{23}^2 +32 b_{13}^2}{33}},$  $b_2 = b_9 =  \pm \sqrt {\frac {11}{4 b_{14}^2 - 4 b_{14} b_{23} + b_{23}^2 + 4 b_{13}^2
} } b_{13} b_{23}$, $b_6 = b_{13}, b_{10} = \mp \frac{1}{11}  \sqrt {\frac{11}{4 b_{14}^2 - 4 b_{14} b_{23} + b_{23}^2 + 4 b_{13}^2
}} (- 13 b_{14} b_{23} + 6 b_{23}^2 + 2 b_{14}^2 + 2 b_{13}^2)$, $b_{28} = \pm \frac{2}{11} \sqrt{\frac{11}{4 b_{14}^2 - 4 b_{14} b_{23} + b_{23}^2 + 4 b_{13}^2
}}(4 b_{14}^2 - 4 b_{14} b_{23} + b_{23}^2 + 4 b_{13}^2)$
\item[(ii)] $a =  2 \sqrt{\frac 23} b_{23}, b_2 = \pm \frac 12 \sqrt{11b_{23}^2 - 4 b_{10}^2}, b_6 = b_{13} = b_{28} = 0, b_{14} = \frac 12 b_{23}.$
\end{enumerate}
For the case $(i)$ we consider  the generic  closed $3$-form $\varphi$  and use all the time that  $a  \neq 0$ and the condition  $ \alpha \wedge \beta =0$  (with  $\alpha=i_{e_7} \varphi$). By imposing the vanishing of  $g_{\varphi} (e_3, e_4)$  and    the condition $g_{\varphi} (e_3, e_3) \neq 0$
we  have always that  either $g_{\varphi} (e_2, e_2) =0$ or  $g_{\varphi} (e_3, e_3) =0$ so we cannot have calibrated $G_2$-structures associated to the Einstein metric.

For the case $(ii)$ we start  only to impose the conditions
$$
a =  \frac{2}{3} \sqrt{6} b_{23}, \qquad b_6 = b_{13} = b_{28} = 0, \qquad b_{14} = \frac 12 b_{23},$$
one needs for the Einstein condition  still to impose that  $b_2 = \pm \frac 12 \sqrt{11b_{23}^2 - 4 b_{10}^2}$.
We  consider  the generic closed $3$-form  $\varphi$,
using  all the time that $b_{23} \neq 0$ (since $a \neq 0$) and we impose  $\alpha \wedge \beta =0$, where $\alpha = i_{e_7} \varphi$. We use that $g_{\varphi} (e_3, e_3 ) \neq 0$ and
 the equations
$$
g_{\varphi} (e_3, e_4)  = g_{\varphi} (e_2, e_3) = g_{\varphi} (e_1, e_3)=0.
$$
Studying separately the solutions of the above system we show that no calibrated $G_2$-structure can induce the Einstein metric.

If $\frak n = [\frak s, \frak s]$ is abelian and $\dim \mathfrak n = 4$, we get the Einstein solvable Lie algebras of rank three with the structure equations

\begin{equation}\label{rank3abel}
\left \{  \begin{array}{l}
d e^1 = a e^{15} + b_1 e^{16} + b_2 e^{17},\\[2 pt]
d e^2 = a e^{25} +  b_3 e^{26} + b_4 e^{27},\\[2 pt]
d e^3 = a e^{35} +  b_5 e^{36} + b_6 e^{37} ,\\[2 pt]
d e^4 = a e^{45}  + b_7 e^{46} + b_8 e^{47} ,\\[2 pt]
d e^5 = d e^6 = d e^7 =0,
\end{array} \right .
\end{equation}

satisfying  the conditions
$$ \begin{array}{l}
 b_7 = -b_1-b_3-b_5, \qquad b_8 = -b_2- b_4-b_6, \\[8pt]
 2 b_1 b_2+2 b_3 b_4 +2 b_5 b_6 + b_2 b_3 + b_2 b_5 + b_1 b_4 + b_4 b_5 + b_1 b_6 + b_3 b_6 =0,\\[8pt]
 2 b_1^2+2 b_3^2+2 b_5^2+2 b_1 b_3 +2 b_1 b_5 +2 b_3 b_5= 4 a^2.
 \end{array}
 $$
We impose for the generic $3$-form $\varphi$ the conditions  $d \varphi =0$ and  $\alpha  \wedge \beta =0$ (with $\alpha = i_{e_7} \varphi$), using all the time that $a \neq 0 $. Then  we consider  the equations
\begin{equation} \label{partialrank3ab}
g_{\varphi} (e_1, e_2) = g_{\varphi} (e_1, e_3) = g_{\varphi} (e_1, e_4)= g_{\varphi} (e_2, e_3) =0.
\end{equation}
By    the conditions $$g_{\varphi} (e_1, e_1) \neq 0,  \quad g_{\varphi} (e_2, e_2)  \neq 0, \quad  g_{\varphi} (e_3, e_3)  \neq 0,  \quad g_{\varphi} (e_4, e_4)  \neq 0$$   we get respectively   $$\rho_{1,2,5}  \rho_{1,3,5}   \rho_{1,4,5} \neq 0,  \,   \rho_{1,2,5}   \rho_{2,3,5}  \rho_{2,4,5}  \neq 0,  \,   \rho_{1,3,5}   \rho_{2,3,5}  \rho_{3,4,5} \neq 0, \,    \rho_{1,4,5}  \rho_{2,4,5}   \rho_{3,4,5} \neq 0.$$
In  all the solutions  of \eqref{partialrank3ab} we   have  always that  either $ \rho_{1,2,5} =0$  or  $ \rho_{1,3,5}  =0$ or  $ \rho_{1,4,5} =0$, which is not possible.
 \end{proof}

So, now we can conclude that there is not a counterexample for the $G_2$ analogue of the Goldberg conjecture in the class of solvmanifolds,  but  we can   show  the existence of  a calibrated $G_2$-structure   whose underlying metric    is a (non trivial)  Ricci soliton.

\begin{example}
If we take the 6-dimensional nilpotent Lie algebra $\mathfrak{n}_6=(0,0,0,0,e^{13}-e^{24},e^{14}+e^{23})$ and we compute the rank-one Ricci soliton solvable extension, we obtain the Lie algebra with structure equations:
$$\mathfrak{s}=(-\frac{1}{2}e^{17},-\frac{1}{2}e^{27}, e^{37}, e^{47}, e^{13}-e^{24}-\frac{1}{2}e^{57}, e^{14}+e^{23}-\frac{1}{2}e^{67}, 0)$$
where the Ricci soliton metric $g$ is the one making the basis  $\{e_1, \dots, e_7\}$ orthonormal.  The Lie algebra $\frak s$  admits the calibrated $G_2$ form
$$\varphi=-e^{136}+e^{145}+e^{235}+e^{246}+e^{567}-e^{127}-e^{347}$$  such that  $g_{\varphi} =g$ and   $d\ast \varphi \neq 0$. Therefore   $g_{\varphi}$  is a Ricci-soliton on $\frak s$  but  $\varphi$ is   not  parallel.
\end{example}


 \section{Cocalibrated $G_2$-structures}\label{sectioncocalibratedstruct}

  In this section we study the existence of cocalibrated $G_2$-structures $\varphi$  on $7$-dimensional Einstein solvable Lie algebras $(\frak s, g)$ whose underlying Riemannian metric $g_{\varphi} = g$.

We will use the classification of the $7$-dimensional Einstein solvable Lie algebras of the previous section, Lemma  \ref{secondobstruction}  and \ref{thirdpobstruction}, together with the following  obstructions.

 \begin{lemma} \label{obstrcocal} Let $(\frak g, g)$ be a $7$-dimensional  metric Lie algebra.
If for every   closed $4$-form $\Psi\in Z^4(\mathfrak{g}^*)$  there exists $X\in \mathfrak{g}$ such that $(i_{X}(*\Psi))^3=0$, then $\mathfrak{g}$ does not admit a cocalibrated $G_2$-structure inducing the metric $g$.
\end{lemma}
\begin{proof} It is sufficient to prove that if a $3$-form $\varphi$ defines a $G_2$-structure  on $\frak g$
then for any $X\in \frak g$ we have
$$
\iota_X(*\varphi)\wedge \varphi\not=0,
$$
where * is the Hodge star operator with respect to the metric $g_{\varphi}$ associated to $\varphi$.
Since the $3$-form $\varphi$ defines a $G_2$-structure on $\frak g$, then
there exists  a basis of $\frak g$
$\{f_1,\cdots, f_7\}$   (which is orthonormal with respect to $g_{\varphi}$)  such that
$$
\varphi =f^{124} + f^{235} + f^{346} + f^{457}+ f^{156}+ f^{267}+ f^{137},
$$
where $\{f^1,\cdots, f^7\}$ is the basis of ${\frak g}^*$ dual to $\{f_1,\cdots, f_7\}$.
Taking the Hodge star operator with respect to $\{f_1,\cdots, f_7\}$ we have
$$
*\varphi = - f^{3567} + f^{1467} - f^{1257} + f^{1236}+ f^{2347} - f^{1345}- f^{2456}.
$$
Thus, writing $\iota_X$ the contraction by $X$,
$$
\iota_{f_1}(*\varphi)=f^{467} - f^{257} + f^{236} - f^{345}
$$
and
$$
\iota_{f_1}(*\varphi)\wedge \varphi= 4 f^{234567} .
$$
Similarly, we have
$$
\begin{array}{l}
\iota_{f_2}(*\varphi)= f^{157} - f^{136}+ f^{347} - f^{456}, \quad
\iota_{f_2}(*\varphi)\wedge \varphi= - 4 f^{134567} ;\\[2 pt]
\iota_{f_3}(*\varphi)= - f^{567} + f^{126} - f^{247} + f^{145}, \quad
\iota_{f_3}(*\varphi)\wedge \varphi= 4 f^{124567} ;\\[2 pt]
\iota_{f_4}(*\varphi)= - f^{167} + f^{237} - f^{135} + f^{256}, \quad
\iota_{f_4}(*\varphi)\wedge \varphi= -4 f^{123567} ;\\[2 pt]
\iota_{f_5}(*\varphi)=  f^{367} - f^{127} + f^{134} - f^{246}, \quad
\iota_{f_5}(*\varphi)\wedge \varphi= 4 f^{123467} ;\\[2 pt]
\iota_{f_6}(*\varphi)=  - f^{357} + f^{147} - f^{123} + f^{245}, \quad
\iota_{f_6}(*\varphi)\wedge \varphi= -4 f^{123457} ;\\[2 pt]
\iota_{f_7}(*\varphi)=  f^{356} - f^{146} + f^{125} - f^{234}, \quad
\iota_{f_7}(*\varphi)\wedge \varphi= 4 f^{123456}.
\end{array}
$$
In general, for $i=1,\cdots 7$, we see that
$$
\iota_{f_i}(*\varphi)\wedge \varphi= (-1)^{i+1} 4 f^{12 \cdots (i-1)(i+1)\cdots 7} ,
$$
which is a non-zero $6$-form for any $i=1, \ldots, 7$.
\end{proof}

\begin{lemma} \label{obstruccocalibrated}
Let $(\mathfrak{g}, g)$ be a  $7$-dimensional metric Lie algebra. If  for any coclosed $3$-form $\varphi$  on $\mathfrak{g}$, the differential form $\tau_3=*d\varphi|_{\Lambda^3_{27}\mathfrak{g}^*}$ satisfies the conditions
\begin{equation*}
\varphi\wedge\tau_3\neq0\\
\text{ or }\\
(*\varphi)\wedge\tau_3\neq0
\end{equation*}
 then $\frg$ does not admit a cocalibrated $G_2$-structure inducing the metric $g$.
\end{lemma}
\begin{proof}The expression of the differential and the codifferential of a $G_2$ form $\varphi$  are given in  terms of  the intrinsic torsion forms  by
\begin{align*}
d\varphi&=\tau_0*\varphi+3\tau_1\wedge\varphi+*\tau_3\\
d*\varphi&=4\tau_1\wedge*\varphi+\tau_2\wedge\varphi.
\end{align*}
with $\tau_0\in\Lambda^0\frg^*$, $\tau_1\in\Lambda^1\frg^*$, $\tau_2\in\Lambda^2_{14}\frg^*$ and $\tau_3\in\Lambda^3_{27}\frg^*$.
The cocalibrated condition $ d * \varphi =0$  implies
$$d\varphi=\tau_0*\varphi+*\tau_3.$$
Since
$$
\begin{array}{l}
\Lambda^3_{27}\frg^*=\{\rho\in\Lambda^3\frg^*|\rho\wedge\varphi=0=\rho\wedge*\varphi\},\\[3 pt]
\Lambda^4_{27}\frg^*=\{\gamma\in\Lambda^4\frg^*|\gamma\wedge\varphi=0=\gamma\wedge*\varphi\}
\end{array}
$$
it follows that
$$d\varphi\wedge\varphi=\tau_0|\varphi|^2e^{1234567}.$$
Therefore
$$\tau_3=-*(d\varphi-\tau_0*\varphi).$$
Now as $\tau_3\in \Lambda^3_{27}\frg^*$ the conditions
$$ \tau_3\wedge\varphi=0, \quad
\tau_3\wedge*\varphi=0
$$
must be fulfilled.
\end{proof}

We recall that a 5-dimensional manifold $N$ has an $SU(2)$-structure if there exists a quadruplet  $(\eta, \omega_1, \omega_2, \omega_3)$, where $\eta$ is a 1-form and $\omega_i$ are 2-forms on $N$ satisfying
$\omega_i \wedge \omega_j=\delta_{ij} v,  \, \,   v\wedge\eta\neq0$
for some nowhere vanishing 4-form $v$, and
$$\iota_{X}\omega_3=\iota_{Y}\omega_1 \implies \omega_2(X,Y)\geq0.$$

\begin{proposition} \label{5dSU(2)}
Let $(\mathfrak{g}=\mathfrak{n}\oplus \mathfrak{a},g)$ be a 7-dimensional Einstein Lie algebra of rank two and let $\{e_1, \dots, e_7\}$ be an orthonormal basis of $(\mathfrak{g},g)$ such that $\mathfrak{a}=\mathbb{R}\langle e_6, e_7\rangle$, then a $G_2$-structure $\varphi$ on $\mathfrak{g}$ induces an $SU(2)$-structure on $\frak n$ such that the associated metric h  is the restriction of $g_{\varphi}$ to $\frak n$.
\begin{proof}
By Lemma \ref{thirdpobstruction}  we know that the forms $F=\iota_{e_7}\varphi$, $\psi_+=\varphi-F \wedge e^7$ determine an $SU(3)$-structure on $\mathbb{R}\langle e_1, \dots, e_6 \rangle$ and that the associated metric is the restriction of $g$ to $\mathbb{R}\langle e_1, \dots, e_6 \rangle$.
Now we can write $F=f^{12}+f^{34}+f^5\wedge e^6$ and $\psi_++\imath\psi_-=(f^1+\imath f^2)\wedge (f^3+\imath f^4)\wedge(f^5+\imath e^6)$ where $f_i \in \mathbb{R}\langle e_1, \dots, e_5 \rangle$ and $\{f_1, \dots, f_5, e_6 \}$ is orthonormal. Then by  \cite[Proposition 1.4]{CS}  the forms
$$\eta=f^5, \hspace{0.3 cm} \omega_1=f^{12}+f^{34}, \hspace{0.3 cm} \omega_2=f^{13}+f^{42}, \hspace{0.3 cm} \omega_3=f^{14}+f^{23}$$
define an $SU(2)$-structure on $\mathfrak{n}$. The basis $\{f_1, \dots, f_5\}$ is orthonormal with respect to the metric $h$ induced by the $SU(2)$-structure. So, $h$ coincides with the restriction of $g_{\varphi}$ to $\mathfrak{n}$.
\end{proof}
\end{proposition}

\begin{corollary} \label{obstrcocal2}
Let $(\mathfrak{g}=\mathfrak{n}\oplus \mathfrak{a},g)$ be a 7-dimensional Einstein Lie algebra of rank two and let $\{e_1, \dots, e_7\}$ be an orthonormal basis such that $\mathfrak{a}=\mathbb{R}\langle e_6, e_7\rangle$. If for any coclosed 3-form $\varphi$ one of the following conditions \begin{itemize}
\item  $(\omega^2_i-\omega^2_j)\wedge\eta \neq 0$ for some $i \neq j$;
\item  $\omega_i\wedge \eta\neq\ast_{h} \omega_i $ for some $i$,
\end{itemize}
holds, where $(\omega_1, \omega_2, \omega_3, \eta, h)$ is the $SU(2)$-structure as in Proposition \ref{5dSU(2)}, then $(\frak g, g)$  does not admit any cocalibrated $G_2$-structure $\varphi$ such that $g_{\varphi}=g$.

\begin{proof}
By Proposition \ref{5dSU(2)} the $G_2$-structure induces an $SU(2)$-structure  $(\omega_1, \omega_2, \omega_3, \eta)$ on $\frak n$. By definition of $SU(2)$-structure the forms  $(\omega_1, \omega_2, \omega_3, \eta)$ have to satisfy the conditions $(\omega^2_i-\omega^2_j)\wedge\eta \neq 0$ for all $i, j$ and $\omega_i\wedge \eta=\ast_{h} \omega_i $ for all $i=1, 2, 3$.
\end{proof}
\end{corollary}

We know already that a $7$-dimensional Enstein solvable Lie algebra cannot admit  nearly-parallel $G_2$-structures since the scalar curvature has to be positive. For the cocalibrated $G_2$-structures
we can prove the following

\begin{theorem}
A $7$-dimensional   (nonflat) Einstein solvmanifold   $(S, g)$   cannot  admit  any  left-invariant    cocalibrated $G_2$-structure $\varphi$ such that $g_{\varphi}= g$.

\end{theorem}

\begin{proof} For a $7$-dimensional rank-one Einstein solvable Lie algebra $(\frak s, g)$  we have   the orthogonal decomposition (with respect to  the Einstein metric  $g$)
 $$
 \frak s = \frak n_6  \oplus \frak  a,
 $$
 with  $\frak n_6 = [\frak s, \frak s]$  a $6$-dimensional nilpotent  Lie algebra and $\frak a = \R \langle e_7\rangle$  abelian.

 If  $\frak n$ is abelian, then we know by \cite[Proposition 6.12]{Heber} that $\mathfrak s$ has structure equations
 $$
(a e^{17}, a e^{27}, a e^{37}, a e^{47}, a e^{57}, a  e^{67},0),
$$
where $a$ is a non-zero real number.
Computing  the generic co-closed $4$-form on $\frak s$  it easy to check that $\frak s$ cannot admit any cocalibrated $G_2$-structure  $g_{\varphi}$ such that  $g_{\varphi} = g$.

 If $\frak n_6$ is nilpotent (non-abelian), then  $(\frak s, g)$  is isometric  to one of the solvable Lie algebras  $\frak g_i, i = 1, \ldots, 33,$ in Table 2, endowed with the Riemannian metric such that the basis $\{e_1, \ldots, e_7\}$ is orthonormal.
 We may apply Lemma \ref{obstrcocal}   with $X =e_7$  to    the  Lie   algebras
 $\frak g_3$, $\frak g_{13}$, $\frak g_{23}$ and   $\frak g_{j}$, $25 \leq j \leq  33$,   showing in this way  that they  do not admit any cocalibrated $G_2$-structure $\varphi$ such that $g_{\varphi} = g$.
 For the   Lie algebras:
 $$\frak g_1, \frak g_2, \frak g_4, \frak g_5, \frak g_6, \frak g_{20}$$
  we   first determine the generic co-closed $3$-form  $\varphi$ and then,  we compute the values of $\tau_0$, and of the $3$-form $\tau_3$. We have that $\tau_3 \wedge \varphi  \neq 0$  unless $\varphi = 0$ and so by applying Lemma \ref{obstruccocalibrated} we have  that  the  Lie algebras do  not admit a cocalibrated  $G_2$-structure
inducing  an Einstein metric. For the Lie algebras
\begin{equation}\label{secondlistcocal}
\frak g_7, \frak g_8, \frak g_9, \frak g_{10}, \frak g_{12}, \frak g_{14}, \frak g_{15}, \frak g_{16}, \frak g_{17}, \frak g_{18}, \frak g_{19}, \frak g_{21}, \frak g_{22}, \frak g_{24}
\end{equation}
  we first determine the generic co-closed $3$-form  $\varphi$ and then, by applying Lemma \ref{thirdpobstruction}, we  impose, that  $\alpha \wedge \alpha \wedge \alpha \neq 0$ and $\alpha \wedge \beta =0$, where $\alpha$ and $\beta$ are given in \eqref{alphabetaexpr}.
 Moreover, we have that the  closed $3$-form $\varphi$ defines a $G_2$-structure if and only if the matrix associated to the symmetric  bilinear form $g_{\varphi}$, with respect to the orthonormal basis $\{e_1, \ldots, e_7\}$,  is positive definite.
 Since the Einstein metric is unique up to scaling,   a  calibrated $G_2$-structure  induces an Einstein metric if and only if the  matrix associated to  the  symmetric  bilinear form $g_{\varphi}$,  with respect to the basis $\{e_1,\dots,e_7\}$,  is a multiple of the identity matrix. By a direct computation we have thus that the Lie algebras
 \eqref{secondlistcocal}
 cannot admit any cocalibrated $G_2$-structures inducing an Einstein metric.

For the $7$-dimensional rank-two Einstein solvable Lie algebras, using the same notation as for the calibrated case we obtain the result for $\frak{k}_1, \frak{k}_3, \frak{k}_5, \frak{k}_{6,1}$ and $\frak{k}_{6,2} $ by the first condition of Corollary \ref{obstrcocal2}. For  the remaining Lie algebras,  i.e.,  for $\frak{k}_4, \frak{k}_{7,1}, \frak{k}_{7,2}, \frak{k}_{7,3}, \frak{k}_{7,4}, \frak{k}_8  $ and the extension of the abelian one, the result follows  by using  the second condition of Corollary \ref{obstrcocal2}.

In   the rank-three  case we have to study the extensions of the Lie algebras $\frak n_4=(0,0,e^{12},0)$ and the four-dimensional abelian Lie algebra. For the first one we consider the structure quations \eqref{rank3nonabel}, then we take a generic   3-form $\varphi$ such that all the coefficients of $e^{ijklm}$  in $d * \varphi$ vanish  except those  of  $e^{13467}$ and $e^{23467}$.  Now if we compute the  inner product  $g_{\varphi}$  induced by $\varphi$  and we impose  the conditions  $g_{\varphi}(e_i, e_i)=g_{\varphi}(e_j, e_j)$ and $g_{\varphi}(e_i, e_j)=0$ for all $i \neq j$ we obtain that $g_\varphi(e_6, e_6)=0$. For the  rank-three Einstein extension of the  abelian Lie algebra  we  consider  the  structure equations \eqref{rank3abel} and  we take a  generic coclosed 3-form $\varphi$.  By imposing  $g_{\varphi}(e_i, e_i)=g_{\varphi}(e_j, e_j)$ and $g_{\varphi}(e_i, e_j)=0$ for all $i \neq j$ we obtain $g_{\varphi}(e_6, e_6)=0$.

\end{proof}

  \newpage

\centerline{
  \begin{tabular}{|c|c|}\hline
    $\frg_1$ & {\tiny{$ \left (\frac{a}{2}e^{17}, ae^{27},
\sqrt{13}ae^{12}+\frac{3}{2}ae^{37},
4ae^{13}+2ae^{47},
 2\sqrt{3}e^{14}+2ae^{23}+\frac{5}{2}ae^{57},
-\sqrt{13}ae^{25}+2\sqrt{3} ae^{34}+\frac{7}{2}ae^{67},0 \right )$}}\\\hline
$\frg_2$ & {\tiny{$\left (-\frac{\sqrt{21}}{42}a e^{17},
-\frac{\sqrt{21}}{14}ae^{27},
 ae^{12}-2\frac{\sqrt{21}}{21}ae^{37},
-2\frac{\sqrt{3}}{3}ae^{13}-5\frac{\sqrt{21}}{42}ae^{47},
 ae^{14}-\frac{\sqrt{21}}{7}ae^{57},
 ae^{34}-ae^{25}-3\frac{\sqrt{21}}{14}ae^{67},
 0 \right)$}}\\\hline
 $\frg_3$& {\tiny{$ \left ( -\frac{\sqrt{14}}{56}a e^{17},
-9\frac{\sqrt{14}}{56}ae^{27},
ae^{12}-5\frac{\sqrt{14}}{28}ae^{37},
\frac{\sqrt{6}}{2}ae^{13}-11\frac{\sqrt{14}}{56}ae^{47},
\frac{\sqrt{6}}{2}ae^{14}-3\frac{\sqrt{14}}{14}ae^{57},
ae^{15}-13\frac{\sqrt{14}}{56}ae^{67},
0 \right) $}} \\\hline
 $\frg_4$&{\tiny{$ \left ( ae^{17},
2ae^{27},
2\sqrt{7}ae^{12}+3ae^{37},
\frac{6\sqrt{154}}{11}ae^{13}+4ae^{47},
2\sqrt{7}ae^{14}+2\frac{\sqrt{1155}}{11}ae^{23}+5ae^{57},
2\frac{\sqrt{1155}}{11}ae^{15}+5\frac{\sqrt{154}}{11}ae^{24}+6ae^{67},
0 \right )$}} \\\hline
 $\frg_5$&{\tiny {$\left ( a e^{17},
3ae^{27},
2\sqrt{14}ae^{12}+4ae^{37},
2\sqrt{15}ae^{13}+5ae^{47},
6\sqrt{2}ae^{14}+6ae^{57},
4\sqrt{2}ae^{15}+2\sqrt{15}ae^{23}+7ae^{67},
0\right)  $} } \\\hline
 $\frg_6$& {\tiny{$ \left (  \frac{a}{2} e^{17},
ae^{27},
\sqrt{10}ae^{12}+\frac{3}{2}ae^{37},
\sqrt{10}ae^{13}+2ae^{47},
\sqrt{10}ae^{23}+\frac{5}{2}ae^{57},
\sqrt{10}ae^{14}+\frac{5}{2}ae^{67},
0 \right )  $}}\\\hline
 $\frg_7$&{\tiny{$\left ( ae^{17},
ae^{27},
4ae^{12}+2ae^{37},
2\sqrt{5}ae^{13}+3ae^{47},
2\sqrt{5}ae^{23}+3ae^{57},
4ae^{14}-4ae^{24}a+4e^{67},
0  \right )$  }}\\\hline
 $\frg_8$& {\tiny{ $\left (ae^{17},
ae^{27},
4ae^{12}+2ae^{37},
2\sqrt{5}ae^{13}+3ae^{47},
2\sqrt{5}ae^{23}+3ae^{57},
4ae^{14}+4ae^{24}a+4e^{67},
0  \right) $}} \\\hline
  $\frg_9$&{\tiny {$\left (\frac{-3}{14}ae^{17},
\frac{-11}{28}ae^{27},
\frac{-3}{7}ae^{37},
\frac{\sqrt{5}}{2}ae^{12}+\frac{-17}{28}ae^{47},
ae^{14}-ae^{23}+\frac{-23}{28}ae^{57},
ae^{34}+\frac{\sqrt{5}}{2}ae^{15}-\frac{29}{28}ae^{67},
0  \right ) $}}\\\hline
 $\frg_{10}$& {\tiny {$ \left (  \frac{4}{9}ae^{17},
ae^{27},
\frac{4}{3}ae^{37},
\frac{2\sqrt{114}}{9}ae^{12}+\frac{13}{9}ae^{47},
\frac{2}{9}\sqrt{190}ae^{14}+\frac{17}{9}ae^{57},
\frac{2\sqrt{114}}{9}ae^{15}+\frac{2\sqrt{114}}{9}ae^{23}+\frac{7}{3}ae^{67},
0\right ) $ }}\\\hline
$\frg_{11}$ & {\tiny {$\left (  \frac{a}{3}e^{17},
\frac{2}{3}ae^{27},
\frac{10\sqrt{7}}{21}ae^{12}+ae^{37},
\frac{4\sqrt{42}}{21}ae^{12}+ae^{47},
\frac{4\sqrt{105}}{21}ae^{13}+\frac{2\sqrt{70}}{21}ae^{14}+\frac{4}{3}ae^{57},
\frac{2\sqrt{6}}{3}ae^{15}+\frac{2\sqrt{7}}{3}ae^{24}+\frac{5}{3}ae^{67},
0 \right )$}} \\\hline
 $\frg_{12}$& {\tiny {$\left (
 \frac{a}{2}e^{17},
ae^{27},
\frac{11}{6}ae^{37},
\frac{2\sqrt{21}}{3}ae^{12}+\frac{3}{2}ae^{47},
\frac{2\sqrt{21}}{3}ae^{14}+2ae^{57},
\frac{2\sqrt{14}}{3}ae^{15}+\frac{2\sqrt{14}}{3}ae^{24}+\frac{5}{2}ae^{67},
0 \right )$ }} \\\hline
 $\frg_{13}$& {\tiny { $ \left (
\frac{a}{9}e^{17},  ae^{27},
\frac{4}{3}ae^{37},
\frac{2\sqrt{93}}{9}ae^{12}+\frac{33}{27}ae^{47},
\frac{4\sqrt{31}}{9}ae^{14}+\frac{39}{27}ae^{57},
\frac{2\sqrt{93}}{9}ae^{15}+\frac{5}{3}e^{67},
0 \right )$ }}\\\hline
   $\frg_{14}$&{\tiny {$\left ( \frac{a}{2}e^{17},
ae^{27},
\frac{3}{4}ae^{37},
\frac{\sqrt{21}}{2}ae^{12}+\frac{3}{2}ae^{47},
\frac{\sqrt{14}}{2}ae^{13}+\frac{5}{4}ae^{57},
\frac{\sqrt{14}}{2}ae^{14}+\frac{\sqrt{21}}{2}ae^{35}+2ae^{67},
0 \right )$}} \\\hline
 $\frg_{15}$&  {\tiny{ $\left (ae^{17},
ae^{27},
ae^{37},
\sqrt{10}ae^{12}+2ae^{47},
\sqrt{10}ae^{23}+2ae^{57},
\sqrt{10}ae^{14}+\sqrt{10}ae^{35}+3ae^{67},
0 \right )$ }} \\\hline
$\frg_{16}$& {\tiny {$\left (   ae^{17},
ae^{27},
ae^{37},
\sqrt{10}ae^{12}+2ae^{47},
\sqrt{10}ae^{23}+2ae^{57},
\sqrt{10}ae^{14}-\sqrt{10}ae^{35}+3ae^{67},
0\right )$ }}\\\hline
 $\frg_{17}$& {\tiny{ $\left ( ae^{17},
ae^{27},
\frac{12}{5}ae^{37},
\frac{4}{5}\sqrt{31}ae^{12}+2ae^{47},
\frac{2}{5}\sqrt{93}ae^{14}+3ae^{57},
\frac{2}{5}\sqrt{93}ae^{24}+3ae^{67},
0 \right )$}} \\\hline
 $\frg_{18}$& {\tiny{ $\left (  ae^{17},  ae^{27},
2ae^{37},
4ae^{12}+2ae^{47},
2ae^{13}-2\sqrt{3}ae^{24}+3ae^{57},
2\sqrt{3}ae^{14}+2ae^{23}+3ae^{67},
0\right )$ }}\\\hline
$\frg_{19}$& {\tiny{$\left (  5ae^{17},
6ae^{27},
12ae^{37},
2\sqrt{134}ae^{12}+11ae^{47},
\sqrt{402}ae^{14}+16ae^{57},
\sqrt{134}ae^{13}-\sqrt{402}ae^{24}+17ae^{67},
0 \right)$}} \\\hline
 $\frg_{20}$& {\tiny{$\left (ae^{17},
ae^{27},
2ae^{12}+2ae^{37},
2\sqrt{3}ae^{12}+2ae^{47},
4ae^{14}+3ae^{57},
2ae^{24}+2\sqrt{3}ae^{23}+3ae^{67},
0 \right )$}} \\\hline
 $\frg_{21}$& {\tiny{$\left (  3ae^{17},
5ae^{27},
6ae^{37},
2\sqrt{42}ae^{12}+8ae^{47},
2\sqrt{21}ae^{13}+9ae^{57},
2\sqrt{42}ae^{14}+2\sqrt{21}ae^{23}+11ae^{67},
0 \right)$ }} \\\hline
  $\frg_{22}$& {\tiny {$\left (   6ae^{17},
5ae^{27},
9ae^{37},
2\sqrt{93}ae^{12}+11ae^{47},
2\sqrt{93}ae^{13}+15ae^{57},
4\sqrt{31}ae^{24}+16ae^{67},
0 \right )$}} \\\hline
 $\frg_{23}$& {\tiny {$\left ( ae^{17},
\frac{5}{2}ae^{27},
3ae^{37},
\sqrt{37}ae^{12}+\frac{7}{2}ae^{47},
\frac{\sqrt{74}}{2}ae^{13}+4ae^{57},
\sqrt{37}ae^{14}+\frac{9}{2}ae^{67},
0\right )$}} \\\hline
 $\frg_{24}$&{\tiny {$\left (ae^{17},
ae^{27},
ae^{37},
\sqrt{6}ae^{12}+2ae^{47},
\sqrt{6}ae^{13}+2ae^{57},
\sqrt{6}ae^{23}+2ae^{67},
0  \right )$}} \\\hline
 $\frg_{25}$& {\tiny {$\left (\frac{5\sqrt{31}}{124}ae^{17},
\frac{2\sqrt{31}}{31}ae^{27},
\frac{9\sqrt{31}}{124}ae^{37},
\frac{9\sqrt{31}}{124}ae^{47},
ae^{12}+\frac{13\sqrt{31}}{124}ae^{57},
-\frac{\sqrt{3}}{2}ae^{34}-\frac{\sqrt{3}}{2}ae^{15}+\frac{9\sqrt{31}}{62}ae^{67},
0 \right )$}} \\\hline
  $\frg_{26}$& {\tiny {$\left (  ae^{17},
2ae^{27},
3ae^{37},
3ae^{47},
4\sqrt{2}ae^{12}+3ae^{57},
4\sqrt{2}ae^{15}+4ae^{67},
0
 \right )$}} \\\hline
   $\frg_{27}$&{\tiny{$\left ( ae^{17},
\frac{3}{4}ae^{27},
\frac{7}{4}ae^{37},
\frac{3}{2}ae^{47},
\frac{\sqrt{148}}{4}ae^{12}+\frac{7}{4}ae^{57},
\frac{\sqrt{74}}{4}ae^{14}+\frac{\sqrt{37}}{2}ae^{25}+\frac{5}{2}ae^{67},
0 \right )$ }} \\\hline
 $\frg_{28}$&{\tiny {$\left ( ae^{17},
ae^{27},
ae^{37},
ae^{47},
2ae^{13}-2ae^{24}+2ae^{57},
2ae^{14}+2ae^{23}+2ae^{67},
0 \right )$}} \\\hline
 $\frg_{29}$& {\tiny {$\left ( ae^{17},
ae^{27},
\frac{4}{3}ae^{37},
\frac{4}{3}ae^{47},
\sqrt{6}ae^{12}+2ae^{57},
\sqrt{6}ae^{14}+\sqrt{6}ae^{23}+\sqrt{7}{3}ae^{67},
0\right )$}} \\\hline
  $\frg_{30}$&{\tiny{ $\left (  \frac{a}{2}e^{17},
\frac{a}{2}e^{27},
\frac{a}{2}e^{37},
\frac{a}{2}e^{47},
\sqrt{2}ae^{12}+ae^{57},
\sqrt{2}ae^{34}+ae^{67},
0\right )$}} \\\hline
 $\frg_{31}$& {\tiny {$\left ( \frac{\sqrt{11}}{11}ae^{17},
\frac{3\sqrt{11}}{22}ae^{27},
\frac{3\sqrt{11}}{22}ae^{37},
\frac{2\sqrt{11}}{11}ae^{47},
\frac{5\sqrt{11}}{22}ae^{57},
ae^{13}+\frac{5\sqrt{11}}{22}ae^{67},
0 \right )$}} \\\hline
 $\frg_{32}$&{\tiny {$\left (  \frac{a}{2}e^{17},
\frac{a}{2}e^{27},
\frac{a}{2}e^{37},
\frac{a}{2}e^{47},
\frac{2}{3}ae^{57},
\frac{\sqrt{11}}{3}ae^{12}+\frac{\sqrt{11}}{3}ae^{34}+ae^{67},
0\right )$}} \\\hline
$\frg_{33}$& {\tiny {$\left ( \frac{a}{2}e^{17},
\frac{a}{2}e^{27},
\frac{3}{4}ae^{37},
\frac{3}{4}ae^{47},
\frac{3}{4}ae^{57},
\frac{\sqrt{34}}{4}ae^{12}+ae^{67},
0 \right )$}} \\\hline
  \end{tabular}
  }

  \smallskip

\centerline{{\bf Table 2.} Rank-one  Einstein  $7$-dimensional  solvable Lie algebras}

\newpage

 \centerline{
  \begin{tabular}{|c|c|}\hline
$\frak s_7$ &  Calibrated $G_2$-structure \\\hline
    $\frg_1$ & \tiny{$\varphi=\frac{1}{432} \left(1440-128 \sqrt{3}\right) e^{123}+\frac{\sqrt{13}}{2}e^{125}-\frac{\left(13312 \sqrt{3}-748800\right)}{44928 \sqrt{3}} e^{127}$}\\
    &\tiny{$+\frac{8}{9} e^{135}-2 e^{137}-\frac{1}{\sqrt{3}}e^{146}-\sqrt{3} e^{147}+10 e^{157}-e^{167}-\frac{1}{3}e^{236}+e^{237}$}\\
   &\tiny{$+\frac{1}{576} \left(1440+128 \sqrt{3}\right)e^{247}+\frac{1}{\sqrt{3}}e^{267}+\frac{1}{\sqrt{3}}e^{345}-e^{357}-e^{457}+e^{567}$}\\\hline
   $\frg_4$ & \small{$\varphi=-\frac{7}{2 \sqrt{5}} e^{125}+e^{137}-\frac{7}{13} e^{146}-e^{147}+\frac{1}{2}e^{167}+\frac{7}{13}e^{236}-e^{237}$}\\
   & \small{$+2e^{247}-e^{267}+\frac{7}{13}e^{345}+\frac{1}{2}e^{357}-e^{457}-e^{567}$} \\\hline
   $\frg_9$ & \small{$\varphi=-\frac{7}{2 \sqrt{5}} e^{125}+e^{137}-\frac{7}{13} e^{146}-e^{147}+\frac{1}{2}e^{167}+\frac{7}{13} e^{236}$}\\
   & \small{$-e^{237}+2 e^{247}-e^{267}+\frac{7}{13}e^{345}+\frac{1}{2}e^{357}-e^{457}-e^{567}$}\\\hline
   $\frg_{18}$ & \small{$\varphi=e^{123}-e^{127}-e^{136}+\sqrt{3} e^{145}$}\\
   &\small{$+3 e^{167}+e^{235}+\sqrt{3} e^{246}-\frac{1}{2}e^{347}+3 e^{567}$}\\\hline
      $\frg_{28}$ & \small{$\varphi=-2 e^{127}-2 e^{347}-e^{136}+e^{145}+e^{235}+e^{246}+2 e^{567}$} \\\hline
    \end{tabular}
  }

  \medskip

  \medskip

  \medskip

\centerline{{\bf Table  3.} Calibrated $G_2$-structures on rank-one Einstein solvable Lie algebras.}

\medskip

\bigskip

\vskip.3cm

\noindent {\bf Acknowledgments.} We are very grateful to the referee for useful comments
that helped to improve the paper. This work has been partially supported
through Project MICINN (Spain) MTM2011-28326-C02-02,
Project UPV/EHU ref.\ UFI11/52 and by  GNSAGA of INdAM.

\medskip

\small\noindent Universidad del Pa\'{\i}s Vasco, Facultad de Ciencia y Tecnolog\'{\i}a, Departamento de Matem\'aticas,
Apartado 644, 48080 Bilbao, Spain. \\
\texttt{marisa.fernandez@ehu.es}\\
\texttt{victormanuel.manero@ehu.es}

\medskip

\small\noindent Dipartimento di Matematica, Universit\`a di
Torino, Via Carlo Alberto 10, Torino, Italy.\\
\texttt{annamaria.fino@unito.it}

\end{document}